\documentclass[10pt, reqno]{amsart}

\usepackage{amssymb,latexsym}
\usepackage{eucal}
\usepackage{tikz}
\usepackage{comment}

\numberwithin{equation}{section}

\newtheorem{thm}{Theorem}[section]

\newtheorem*{TA}{Theorem A}

\newtheorem{lemma}{Lemma}[section]

\newcommand{\al}{\alpha}

\newcommand{\omd}{\omega_{j,D}}
\newcommand{\om}{\omega}

\newcommand{\A}{\mathcal{A}}
\newcommand{\ph}{\varphi}

\newcommand{\PP}{\mathcal{P}}

\newcommand{\wtl}{\widetilde}
\newcommand{\wht}{\widehat}
\DeclareMathOperator{\core}{core}

\DeclareMathOperator{\SLI}{\textnormal{SLI}}
\DeclareMathOperator{\inq}{\textnormal{inq}}
\DeclareMathOperator{\sol}{\textnormal{sol}}

\newcommand{\brr}{\bar {\mathrm{r}}}
\newcommand{\rr}{\mathrm{r}}

\begin{document}

\title
[ Intersecting subgroups in free groups and  linear programming]
{The intersection of  subgroups in free groups and linear programming}
 \author{S. V. Ivanov }
  \address{  Department of Mathematics\\
 University of Illinois \\
  Urbana\\  IL 61801\\ U.S.A. } \email{ivanov@illinois.edu}
\thanks{Supported in smaller part by  the NSF under  grant  DMS 09-01782.}

\begin{abstract}
We study the intersection of finitely generated  subgroups of free  groups by utilizing the method of linear programming.
We prove that if $H_1$ is a finitely generated
subgroup of a free group $F$, then the WN-coefficient $\sigma(H_1)$ of $H_1$ is rational and can be computed in deterministic exponential time {}in the size of $H_1$. This coefficient $\sigma(H_1)$ is the  minimal nonnegative real number such that, for every  finitely generated  subgroup $H_2$ of  $F$, it is true that
$\bar \rr(H_1, H_2)  \le  \sigma(H_1) \bar \rr(H_1) \bar \rr(H_2)$, where  $\bar{ {\rm r}} (H) := \max ( {\rm r} (H)-1,0)$ is the reduced rank of $H$, $\rr (H)$ is the rank of $H$, and  $\bar \rr(H_1, H_2)$ is the reduced rank of the generalized intersection of $H_1$ and $H_2$.
We also show the existence of a  subgroup $H_2^* = H_2^*(H_1)$ of $F$ such that  $\bar \rr(H_1, H_2^*)  =  \sigma(H_1) \bar \rr(H_1) \bar \rr(H_2^*)$, the Stallings graph $\Gamma(H_2^*)$ of  $H_2^*$ has at most doubly exponential size {}in the size of $H_1$ and   $\Gamma(H_2^*)$  can be constructed in  exponential time {}in the size of $H_1$.
\end{abstract}

\maketitle
\section{Introduction}

Let $F$ be a finitely generated free group, let  $\rr(F)$ denote the rank of $F$ and let
$\brr(F) := \max (\rr(F)-1,0)$ denote the reduced rank of $F$.
Let $H_1$ and $H_2$ be finitely generated subgroups of  $F$.  Hanna Neumann \cite{N1}  proved
that
$$
\brr (H_1 \cap H_2) \le 2 \bar \rr(H_1) \bar \rr(H_2)
$$
and conjectured that  $\brr (H_1 \cap H_2) \le  \bar \rr(H_1) \bar \rr(H_2)$.
\smallskip

These result and conjecture of Hanna Neumann were strengthened by Walter Neumann \cite{N2}  by
considering a generalized intersection of $H_1$ and $H_2$.
Let  $S(H_1, H_2)$ denote a set of  representatives of those double cosets $H_1 t H_2$ of $F$, $t \in F$,
that have the property $H_1 \cap t H_2 t^{-1} \ne  \{ 1 \}$. Walter Neumann \cite{N2}  proved
that the set $S(H_1, H_2)$ is finite,  the reduced rank $\bar \rr(H_1, H_2)$ of the  generalized intersection
of $H_1$ and $H_2$     satisfies
\begin{equation}\label{res}
\bar \rr(H_1, H_2) :=   \sum_{s \in S(H_1, H_2)} \bar \rr(H_1\cap s H_2 s^{-1})
\le  2 \bar \rr(H_1) \bar \rr(H_2) ,
\end{equation}
and he conjectured that
\begin{equation}\label{conjs}
\bar \rr(H_1, H_2) =   \sum_{s \in S(H_1, H_2)} \bar \rr(H_1\cap s H_2 s^{-1})
\le   \bar \rr(H_1) \bar \rr(H_2) .
\end{equation}
This strengthened version of the  Hanna Neumann conjecture  was proved by
Friedman \cite{Fr} and Mineyev \cite{Mn}, see also Dicks's proof \cite{Dp} and a proof in \cite{I15a}.
\smallskip

Now suppose that $H_1$ is a fixed finitely generated subgroup of $F$.
We will say that a real number $\sigma(H_1) \ge 0$ is the {\em Walter Neumann coefficient} for  $H_1$, or,  briefly, the WN-coefficient for $H_1$,   if, for every  finitely generated  subgroup $H_2$ of $F$, we have
\begin{equation}\label{e1}
\brr (H_1 , H_2)   \le \sigma(H_1)  \brr(H_1) \brr(H_2)
\end{equation}
and  $\sigma(H_1)$ is minimal  with this property.
It is clear that if $H_1$ is noncyclic then
$$
\sigma(H_1)  = \sup_{H_2} \bigg\{  \frac{\brr (H_1,  H_2)}{\brr (H_1) \brr (H_2) } \bigg\}
$$
over all  finitely generated  noncyclic subgroups $H_2$ of $F$.
\smallskip

In this article, we are concerned with algorithmic computability of the  WN-coefficient $\sigma(H_1)$ for a
finitely generated subgroup $H_1$ of $F$ and with other properties of this number $\sigma(H_1)$.
Utilizing the method of linear programming, we will prove the following.

\begin{thm}\label{th1} Suppose that $F$ is a free group of finite rank
and $H_1$ is a finitely generated noncyclic subgroup of $F$. Then the following  are true.
\smallskip

$\rm{(a)}$ There exists a linear programming problem (LP-problem) associated with $H_1$
\begin{equation}\label{lpa}
\PP(H_1) = \max\{ cx \mid Ax \le b  \}
\end{equation}
with integer coefficients whose solution is equal to $-\sigma(H_1) \brr (H_1)$.
\smallskip

$\rm{(b)}$  There is a finitely generated  subgroup $H_2^* $ of $F$, $H_2^*= H_2^*(H_1)$ which corresponds to  a vertex
solution of the dual problem
$$
\PP^*(H_1) = \min \{ b^{\top}  y \mid A^{\top}y = c^{\top} , \, y \ge 0  \}
$$
of the primal LP-problem  \eqref{lpa} such that
$\bar \rr(H_1, H_2^*)  =  \sigma(H_1)
\bar \rr(H_1) \bar \rr( H_2^*)$. In particular,  the WN-coefficient $\sigma(H_1)$ of $H_1$ is rational and satisfies
$\frac{1}{\brr (F)} \le \sigma(H_1) \le 1$.

Furthermore, if $\Gamma(H_1)$ and $\Gamma(H_2^*)$ denote the Stallings graphs representing the subgroups $H_1$ and $H_2^*$, resp.,
$| E \Gamma |$ denotes the number of oriented edges in the graph $\Gamma$, and  $m-1$ is the rank of $F$,  then
the size of $\Gamma(H_2^*)$ is at most doubly exponential {}in the size of $\Gamma(H_1)$, specifically,
$$
| E  \Gamma(H_2^*) | < 2^{  2^{ | E  \Gamma(H_1) |/2 + 2\log_2 m  } } .
$$

$\rm{(c)}$ Assume that $H_1$ is given by a finite generating set or by its Stallings graph.
Then, in deterministic exponential time {}in the size of the input, one can write down and solve the  LP-problem
\eqref{lpa} associated with $H_1$.  In particular, the WN-coefficient $\sigma(H_1)$ of $H_1$ is computable in deterministic exponential time  {}in the size of the input.

In addition,  the Stallings graph $\Gamma(H_2^*)$ of the subgroup $H_2^*$ of part (b) can be constructed
 in deterministic exponential time {}in the size of the input.
\end{thm}

We remark  that the results similar to those of Theorem~\ref{th1} are obtained by the author \cite{I15} for factor-free subgroups of free products of finite groups. However, the arguments of \cite{I15} do not apply to free products of infinite groups and here we develop analogous techniques suitable for free groups. For a generalization of the conjecture \eqref{conjs} to subgroups of free products of groups and relevant results, the reader is referred to articles
\cite{DIv}, \cite{DIv2}, \cite{Iv08}, \cite{I15a}, \cite{I15}.
\smallskip

Similarly to \cite{I15},  the correspondence
between subgroups $H_2$ and vectors $y(H_2)$ of the feasible polyhedron  $\{ y  \mid A^{\top}y = c^{\top} , \, y \ge 0  \}$  of the dual LP-problem $\PP^*(H_1)$, mentioned in part (b) of Theorem~\ref{th1}, plays an important role in proofs and is reminiscent of the correspondence between (resp. almost) normal surfaces in 3-dimensional manifolds and their (resp. almost)  normal vectors in the Haken theory of normal surfaces and its generalizations, see \cite{Haken}, \cite{HLP}, \cite{Hemion}, \cite{Iv08s}, \cite{JT}. In particular, the idea of a vertex solution works equally well both in the context of almost normal surfaces \cite{Iv08s}, see also \cite{HLP}, \cite{JT}, and in the context of subgroups of  free groups, providing in either  situation  both the connectedness of the underlying object associated with a vertex solution and an upper bound on the size of the underlying  object.
\smallskip

It is worthwhile to mention that our construction of the Stallings graph $\Gamma(H_2^*)$ in part~(c) of   Theorem~\ref{th1} is somewhat succinct (cf. the definition of succinct representations of  graphs, see \cite{PCC}) in the sense that,
despite the fact that the size of $\Gamma(H_2^*)$  could be doubly exponential, we are able to  give a description of $\Gamma(H_2^*)$  in exponential time. In particular, the vertices of  $\Gamma(H_2^*)$ are represented by exponentially long bit strings and the edges of $\Gamma(H_2^*)$ are  drawn in blocks. As a result,
we can find out in exponential time  whether two given vertices of  $\Gamma(H_2^*)$ are connected by an edge.
\smallskip

In view of Theorem~\ref{th1}, it is of interest to look at two properties of finitely generated subgroups of  free groups
introduced by Dicks and Ventura \cite{DV}. Recall that a finitely generated subgroup $H$ of a free group $F$
is called {\em compressed}, see \cite{DV}, if, for every subgroup $K$ of $F$ such that $H  \subseteq K$, we have $\brr(H) \le \brr(K)$.
A finitely generated subgroup $H$ of a free group $F$
is called {\em inert}, see \cite{DV}, if for every subgroup $K$ of $F$, one has $\brr(H\cap K ) \le \brr(K)$.
It is immediate from the definitions that every inert subgroup is compressed.  The problem whether every compressed subgroup is inert is stated by  Dicks and Ventura \cite{DV} and it is still unresolved.
\smallskip

We say that a finitely generated subgroup $H$ of a free group $F$
is  {\em strongly inert} if, for every subgroup $K$ of $F$, we have  $\brr(H, K ) \le \brr(K)$.
Clearly, a strongly inert subgroup is inert. It would be of interest to  find an example, if it exists, to distinguish between these two classes of inert and strongly inert subgroups and, more generally, to find a finitely generated noncyclic subgroup $H$ of  $F$ such that
$$
\sup_{K}  \bigg\{ \frac{\brr (H \cap  K)}{\brr (H) \brr (K) } \bigg\} < \sigma(H) ,
$$
where the supremum, as before, is taken over all
finitely generated noncyclic subgroups $K$ of  $F$. Another natural question is to find an algorithm that computes this number
$\sup_{K} \Big\{ \frac{\brr (H \cap  K)}{\brr (H) \brr (K) } \Big\}$ which could be called the Hanna Neumann coefficient of $H$.
\smallskip

While we are not able to distinguish between these three classes of compressed, inert, and strongly inert subgroups of $F$, our
algorithms and their running times, that recognize two of these classes, are quite different.

\begin{thm}\label{pr1} Suppose that $F$ is a free group of finite rank
and $H$ is a finitely generated noncyclic subgroup of $F$ given by a finite generating set  or by its Stallings graph.
Then the following hold true.

$\rm{(a)}$  There is an algorithm that decides, in deterministic exponential time  {}in the size of $H$, whether $H$ is strongly inert.

$\rm{(b)}$  There is an algorithm that verifies, in nondeterministic polynomial time  {}in the size of $H$, whether $H$ is {\em not} compressed.
\end{thm}

Summarizing, we see that the decision problem that inquires whether a finitely generated subgroup $H$ of $F$ is strongly inert is in \textsf{EXP}, the decision problem that  asks whether  $H$  is inert is not known to be decidable, and
the decision problem that inquires whether  $H$  is compressed is in \textsf{coNP}  (for the definition of computational complexity classes \textsf{EXP}, \textsf{coNP} see \cite{AB} or \cite{PCC}).

\section{Preliminaries}

Suppose that $X$ is a  graph. Let $VX$ denote the set of vertices of $X$ and let $E X$ be the set of oriented edges of $X$.  If $e \in EX$  then $e^{-1}$ denotes the edge  with the opposite orientation, $e^{-1} \ne e$.
\smallskip

For $e \in EX$, let $e_-$ and $ e_+$ denote the initial and terminal, resp., vertices of $e$. A path $p = e_1 \cdots e_k$, where $e_i \in EX$,
$(e_i)_+ =  (e_{i+1})_-$, $i =1, \dots, k-1$, is called {\em reduced } if, for every  $i =1, \dots, k-1$, one has $e_i \ne e_{i+1}^{-1}$. The {\em length} of a path $p = e_1 \cdots e_k$ is $k$, denoted $|p| =k$.  The initial vertex of $p$ is  $p_- = (e_1)_-$ and the terminal vertex of $p$ is  $p_+ = (e_k)_+$. A path $p$ is {\em closed } if $p_- = p_+$.
\smallskip

If $p = e_1 \cdots e_k$ is a closed path, then a {\em cyclic permutation}
$\bar p$ of $p$ is a path of the form $e_{1+i}e_{2+i}  \cdots e_{k+i}$, where $i =0,1,\dots,k$  and the indices are considered $\mod k$.
The subgraph of $X$ that consists of edges of all closed paths $p$ of $X$ such that  $|p| >0$ and every cyclic permutation
of $p$ is reduced,  is called the {\em core} of $X$, denoted $\core(X)$.
\smallskip

Let $U$ be a finite connected graph such that  $\core(U) = U$, let  $o \in VU$ and let $F = \pi_1(U, o)$ be the fundamental group of $U$ at $o$. Then $F$ is a free group of rank $\rr(F) = |  E U |/2 - | VU | +1$, where $| B |$ is the cardinality of a finite set $B$, and the  elements of $F$  can be thought of as reduced closed paths in $U$ starting at $o$.
\smallskip

Following Stallings \cite{St}, see also  \cite{D}, \cite{KM}, with every (finitely generated) subgroup $H$ of
$F_U = \pi_1(U, o)$, we can associate a (resp. finite) graph $Y = Y(H)$ and a map $ \beta : Y \to U$ of graphs so that $H$ is isomorphic to $\pi_1(Y, o_{Y})$, where $o_{Y} \in VY$, $\beta(o_{Y}) = o$, and a reduced path $p \in \pi_1(U, o)$ belongs to $H$ if and only if there  is a reduced path $p_H \in
\pi_1(Y, o_{Y})$ such that   $\beta(p_H) = p$.
In addition, we may assume that $\beta$ is a locally injective map of graphs, i.e., the restriction of $\beta$ on  a regular neighborhood of every vertex of $Y$ is injective.
We call a locally injective map of graphs an {\em immersion}.
Since $\beta$ is an immersion, it follows that every reduced path in $H \subseteq \pi_1(U, o)$  has a unique preimage in $Y$.
\smallskip

Consider two finitely generated subgroups $H_1$ and $H_2$ of the free group
$F_U = \pi_1(U, o)$. Pick a set $S(H_1,H_2)$ of representatives of those double cosets $H_1 g H_2$,
$g \in F_U$, for which the intersection $H_1 \cap g H_2 g^{-1}$ is nontrivial.
\smallskip

Let $Y_1, Y_2$ be Stallings graphs of the subgroups $H_1, H_2$ and let $Y_1 \underset U \times  Y_2$
denote the  pullback of the maps
$\beta_i : Y_i \to U$, $i =1,2$. Recall that
\begin{gather*}
V(Y_1 \underset U \times  Y_2) = \{ (v_1, v_2) \mid  v_i \in VY_i, \beta_1(v_1) =  \beta_2(v_2)  \} , \\
E(Y_1 \underset U \times  Y_2) = \{ (e_1, e_2) \mid  e_i \in EY_i, \beta_1(e_1) =  \beta_2(e_2)  \} ,
\end{gather*}
and $(e_1, e_2)_- = ((e_1)_-, (e_2)_-)$, \ $(e_1, e_2)_+ = ((e_1)_+, (e_2)_+)$.
\smallskip

According to Walter Neumann \cite{N2}, the set $S(H_1, H_2)$ is finite and the  nontrivial intersections  $H_1 \cap s H_2 s^{-1}$, where $s \in S(H_1, H_2)$, are in bijective correspondence with the connected components $W_s$ of the core $W := \core(Y_1 \underset U \times  Y_2)$. Moreover, for every $s \in S(H_1,H_2)$,
we have
\begin{gather*}
\brr ( H_1 \cap s H_2 s^{-1} ) = \brr (W_s) = |  E W_s|/2 - | V W_s | .
\end{gather*}
Hence,
\begin{gather*}
\sum_{s \in S(H_1,H_2)} \brr (H_1 \cap s H_2 s^{-1} ) = \brr (W) = | E W|/2 - |  V W |  .
\end{gather*}

Let $\al'_i$ denote the projection map $Y_1 \underset U \times  Y_2 \to Y_i$,
$i = 1,2$, i.e., $\al'_i((e_1, e_2)) = e_i$ and $ \al'_i((v_1, v_2)) = v_i$. Restricting  $\al'_i$ on $W \subseteq  Y_1 \underset U \times  Y_2 $, we obtain the map  $\al_i : W \to Y_i$, $i = 1,2$. In this notation, we have a commutative diagram depicted in Figure~1.

\par\smallskip\par\bigskip \centerline{
\setlength{\unitlength}{.55mm}
 \begin{picture}(50,38)
  \put(0,16){\makebox(0,0)[t]{$Y_1$}}
  \put(6,32){\makebox(0,0)[t]{$\alpha_1$}}
   \put(5,4){\makebox(0,0)[t]{$\beta_1$}}
    \put(33,32){\makebox(0,0)[t]{$\alpha_2$}}
   \put(36,4){\makebox(0,0)[t]{$\beta_2$}}
   \put(20,38){\makebox(0,0)[t]{$W$}}
  \put(21,-7){\makebox(0,0)[t]{$U$}}
 \put(41,16){\makebox(0,0)[t]{$Y_2$}}
    \thinlines
  \put(14,30){\vector(-1,-1){12}}
   \put(25,30){\vector(1,-1){12}}
  \put(5,8){\vector(1,-1){12}}
  \put(35,8){\vector(-1,-1){12}}
       \put(9,-22){Figure~1}
\end{picture}
}\par\vspace{16mm}

In particular, if $X \in \{ Y_1, Y_2, W, U \}$, then there is a canonical immersion
$$\ph : X \to U ,
$$
where $\ph = \beta_i$ if $X = Y_i$, $i =1,2$,
$\ph = \mbox{id}_U$ if $X = U$ and  $\ph =  \beta_i\al_i  $ if $X = W$.
More generally, we will say that $X$ is a {\em $U$-graph} if $X$ is equipped with a graph map $\ph : X \to U$. A $U$-graph  $X$ is {\em reduced} if  $\ph$ is an immersion. For example,  $Y_1, Y_2, W, U$ are reduced $U$-graphs.
If $x \in VX \cup EX$  then $\ph(x) \in VU \cup EU$ is called the {\em label} of $x$.
\smallskip

It will be convenient to work with a graph $U$ of a special form which we denote $U_m$. The graph $U_m$  contains two vertices $o_1, o_2$,
 $VU_m := \{ o_1, o_2 \}$, and the vertices $o_1, o_2$ are connected by $m \ge 3$ nonoriented edges so that the oriented edges $a_1, \dots, a_m \in  E U_m$   start at $o_1$ and end in  $o_2$, see Figure~2,  where the case $m=3$ is depicted.

\begin{center}
\begin{tikzpicture}[scale=.92]
\draw  (0.1,0) ellipse (1 and 1);
\draw[-latex](0,0)--(0.2,0);
\draw(-0.9,0)--(1.1,0);
\draw  (-0.9,0)[fill = black]circle (0.07);
\draw  (1.1,0)[fill = black]circle (0.07);
\draw[-latex](0,1) -- (0.2,1);
\draw[-latex](0,-1) -- (0.2,-1);
\node at (0.1,0.36) {$a_2$};
\node at (0.1,-0.64) {$a_3$};
\node at (0.1,1.36) {$a_1$};
\node at (1.7,1) {$U_3$};
\node at (0.1,-1.84) {Figure 2};
\end{tikzpicture}
\end{center}

We will be considering mostly $U_m$-graphs, where $m \ge 3$.  Since $U_m$ is fixed, we will be writing  $Y_1 \times  Y_2$ in place of  $Y_1 \underset {U_m} \times  Y_2$.
\smallskip

Denote
\begin{gather}\label{AL}
\A := \{ a_1, \dots, a_m \} .
\end{gather}

Let $X$ be a $U_m$-graph.
A vertex $v \in VX$ is called an {\em $i$-vertex} if $\ph(v) = o_i$, $i =1,2$.
Clearly, if $e_+$ is a 2-vertex then
$ \ph(e) \in \A  $
and if $e_+$ is a 1-vertex then  $\ph(e) \in  \A^{-1} = \{ a_1^{-1}, \dots, a_m^{-1} \}$.
\smallskip

An edge $e \in EX$ is called a $b$-edge, where $b \in \A$, if $\ph(e) = b$. The set of all  $b$-edges of $X$ is denoted $E_{b}X$. Clearly,
$$
\sum_{b \in \A} |E_b X| = |E X|/2 .
$$

\section{The system of linear inequalities $\SLI[Y_1]$}

Suppose that  $Y_1$ is a finite reduced $U_m$-graph such that $Y_1 = \core(Y_1)$ and
$$
\brr(Y_1) := - \chi(Y_1) = |E Y_1|/2 - |V Y_1| >0,
$$
where $\chi(Y_1)$ is the Euler characteristic of $Y_1$ (since $E Y_1$ is the set of oriented edges of $Y_1$, we use $|E Y_1|/2$ in $\chi(Y_1)$).  This graph $Y_1$ will be held fixed throughout Sections~3--4.
\smallskip

Let $(A_1, \dots, A_m)$  be an $m$-tuple of sets $A_j$ such that
\begin{gather}\label{a00}
A_j \subseteq E_{a_j}Y_1 ,  \quad j = 1, \dots, m .
\end{gather}
Recall that $m = | \A |$ and $m \ge 3$, see \eqref{AL}.
\smallskip

Let $e, f \in \bigcup_{j=1}^m A_j$. We say that the edges
$e, f$ are {\em $i$-related}, written $e \sim_i f$, if $e_- = f_-$ in $Y_1$ when $i=1$ or $e_+ = f_+$ in $Y_1$ when $i=2$.
Note that it follows from  \eqref{a00} and from $Y_1$ being a $U_m$-graph  that $e_-, f_-$ are 1-vertices while
$e_+, f_+$ are 2-vertices. Clearly, ${\sim_i}$ is an equivalence relation on the set $\bigcup_{j=1}^m A_j$.
\smallskip

Let $[e]_{\sim_i}$ denote the equivalence class of an edge $e \in \bigcup_{j=1}^m A_j$  relative to this
equivalence relation ${\sim_i}$ and let $| [e]_{\sim_i} |$ denote the cardinality of $[e]_{\sim_i}$.
\smallskip

We will say that an $m$-tuple  $(A_1, \dots, A_m)$ is {\em $i$-admissible}, where $i =1,2$ is fixed, if
the union $\bigcup_{j=1}^m A_j$ is not empty and, for every $e \in \bigcup_{j=1}^m A_j$, we have $ | [e]_{\sim_i} | >1$.
It is clear that, for every $e \in \bigcup_{j=1}^m A_j$,
\begin{gather}\label{a1}
2 \le | [e]_{\sim_i} | \le k \le m ,
\end{gather}
where $k$ is the number of nonempty sets  $A_j$ in the tuple $(A_1, \dots, A_m)$.
\smallskip

If  $(A_1, \dots, A_m)$ is an $i$-admissible tuple, we define the number  $N_i(A_1, \dots, A_m)$ to be the sum
$$
\sum ( | [e]_{\sim_i} | -2)
$$
over all equivalence classes $[e]_{\sim_i}$ of the equivalence  relation ${\sim_i}$
on $\bigcup_{j=1}^m A_j$. Hence,
\begin{gather}\label{a2}
N_i(A_1, \dots, A_m) :=  \sum_{[e]_{\sim_i}} ( | [e]_{\sim_i} | -2) .
\end{gather}

We note that
$$
\brr(Y_1) = | E Y_1 |/2 - |V Y_1 | = \tfrac 12  \sum_{u \in VY_1} (\deg u -2) ,
$$
where $\deg u$ is the degree of a vertex $u \in VY_1$, i.e.,
$\deg u$ is the number of edges $e \in EY_1$ such that $e_+ = v$.
\smallskip

Let $V_iY_1$ denote the set of all $i$-vertices of $Y_1$, $i=1,2$.  Define
$$
\brr_i(Y_1) := \tfrac 12  \sum_{u \in V_iY_1} (\deg u -2) .
$$
Observe that $\brr(Y_1)= \brr_1(Y_1)+\brr_2(Y_1)$ and that
\begin{align}\label{a3a}
\begin{split}
N_i(A_1, \dots, A_m) & :=  \sum_{[e]_{\sim_i}} ( | [e]_{\sim_i} | -2) \\
 & \le \sum_{u \in V_iY_1} (\deg u -2) = 2 \brr_i(Y_1) \le  2 \brr(Y_1) .
\end{split}
\end{align}

For every nonempty set $B \subseteq E_{a_j} Y_1$, we
consider a variable $x_{j, B}$.
We also introduce a special variable $x_s$.
Note that, for given $j$, the set of  all variables $x_{j, B}$ is finite and
its cardinality is equal to $2^{| E_{a_j} Y_1 |}-1$.

\smallskip

Now we will define a system of inequalities in these variables $x_{j, B}$, $x_s$ so that each inequality
is determined by means of an $i$-admissible tuple $(A_1, \dots, A_m)$.
\smallskip

For an $i$-admissible tuple $(A_1, \dots, A_m)$, let $A_{j_1}, \dots, A_{j_k}$ denote all nonempty sets
in  $(A_1, \dots, A_m)$.
\smallskip

If $i=1$ then the inequality, corresponding to the $1$-admissible tuple $(A_1, \dots, A_m)$, is defined as follows
\begin{equation}\label{a3}
- x_{j_1, A_{j_1}}   - x_{j_2, A_{j_2}}   - \ldots - x_{j_k, A_{j_k}}  - (k-2)x_s \le   - N_1(A_1, \dots, A_m)  .
\end{equation}

If $i=2$  then  the inequality, corresponding to the $2$-admissible tuple $(A_1, \dots, A_m)$, is defined as follows
\begin{equation}\label{a4}
x_{j_1, A_{j_1}}   + x_{j_2, A_{j_2}}   + \ldots + x_{j_k, A_{j_k}}  - (k-2)x_s \le  - N_2(A_1, \dots, A_m)  .
\end{equation}

Let
\begin{equation}\label{SLI}
\SLI[Y_1]
\end{equation}
denote  the system of all linear inequalities \eqref{a3}--\eqref{a4} constructed for all $i$-admissible tuples
$(A_1, \dots, A_m)$, $i =1,2$. Clearly, $\SLI[Y_1]$ is finite.
\medskip

Assume that the map
$$
\al_2 : \core(Y_1  \times  Y_2) \to Y_2
$$
is surjective, i.e.,
$\al_2 ( \core(Y_1  \times  Y_2) ) = Y_2$.
\smallskip

For every $i$-vertex $u \in VY_2$,  consider all the edges  $e_1, \dots, e_k \in EY_2$ such that   $(e_1)_-=  \cdots = (e_k)_-= u$ if $i=1$  or  $(e_1)_+=  \cdots  = (e_k)_+= u$ if $i=2$,
so  $\ph(e_1), \ldots, \ph(e_k)   \in \A$, see \eqref{AL},  and
the edges  $e_1, \ldots, e_k$ start or finish at  $u$.  Denote $\ph(e_\ell) = a_{j_\ell}$ for $\ell =1, \ldots, k$.
\smallskip

If $j \not\in \{ j_1, \dots, j_k\}$, we set  $A_j(u) := \varnothing$.
Otherwise, $j = j_\ell$ for some $\ell = 1, \ldots, k$,  and we set
\begin{equation}\label{Adu}
A_{j_\ell}(u) := \al_1 \al_2^{-1}  (e_{\ell}) \subseteq E_{j_\ell}Y_1 ,
\end{equation}
where $\al_2^{-1}  (e_{\ell})$ is the full preimage of the edge $e_{\ell}$ in  $\core(Y_1  \times  Y_2)$. It is immediate from the definitions that the tuple  $(A_1(u), \dots, A_m(u))$ is $i$-admissible. Recall that the  graphs $Y_1$ and $Y_2$
coincide with their cores and have no vertices of degree less than 2.
\smallskip

Since every $i$-admissible tuple $(A_1, \dots, A_m)$  gives rise  to an inequality
\eqref{a3}  if $i=1$ or to an inequality \eqref{a4} if $i=2$ and every $i$-vertex $u \in VY_2$ defines, as indicated above,
an $i$-admissible tuple $(A_1(u), \dots, A_m(u))$, it follows that every vertex
$u \in VY_2$ is mapped to a certain inequality of the system  $\SLI[Y_1]$, denoted $\inq_V(u)$. Thus we obtain a function
$$
\inq_V : V Y_2 \to  \SLI[Y_1]
$$
from the set $V Y_2$ of vertices  of a graph $Y_2$, with the property that the map
$$
\al_2 :  \core(Y_1  \times  Y_2) \to  Y_2
$$  is surjective,
to the set of inequalities of the system $\SLI[Y_1]$.
\smallskip

If $q$ is an inequality of the system $\SLI[Y_1]$, written $q \in \SLI[Y_1]$, we let $q^L$ denote the left hand side of $q$,   let $q^R$ denote the number of the right hand side of the inequality $q$, and  let $k(q) \ge 2$ denote the parameter $k$ for $q$,  see the definition of inequalities \eqref{a3}--\eqref{a4}.
\smallskip

\begin{lemma}\label{lem1}    Suppose $Y_2$ is a finite reduced $U_m$-graph with the property that
the map   $\al_2 : \core(Y_1  \times  Y_2) \to Y_2$ is surjective. Then
 \begin{equation*}
\sum_{u \in VY_2}  \inq_V(u)^L = -2 \brr (Y_2) x_s    \quad  \mbox{and}  \quad
\sum_{u \in VY_2}  \inq_V(u)^R = -2 \brr ( \core(Y_1  \times  Y_2)  )  .
\end{equation*}
\end{lemma}

\begin{proof} Suppose $e \in EY_2$,  $\ph(e) =a_j$,  and $e_- = u_1$, $e_+ = u_2$. Clearly, $u_i$ is an $i$-vertex of $Y_2$, $i=1,2$. Denote $B := \al_1 \al_2^{-1} (e) \subseteq  E_{a_j}Y_1$.  Then the variables $-x_{j, B}$ and $x_{j, B}$ of  $\inq_V(u_1)^L$ and $\inq_V(u_2)^L$, resp., will cancel out in the sum  $\sum_{u \in VY_2}  \inq_V(u)^L$.  It is easy to see that  all occurrences  of the variable $\pm x_{j, B}$, where $j =1, \dots, m$, $B \subseteq E_{a_j}Y_1$, $|B| >0$, in the sum  $\sum_{u \in VY_2}  \inq_V(u)^L$ can be paired down by using edges of $Y_2$  as indicated above.
\smallskip

Now we observe that every vertex $u \in VY_2$ of degree $d \ge 2$ contributes $-(d-2)$ to the coefficient of $x_s$ in the sum $\sum_{u \in VY_2}  \inq_V(u)^L$ and that
$$
-\chi(Y_2) = \brr (Y_2) = \tfrac 12 \sum_{u \in VY_2} (\deg u -2) .
$$

Therefore, we may conclude that
$$
\sum_{u \in VY_2}  \inq_V(u)^L =  -2 \brr (Y_2) x_s ,
$$
as required.
\smallskip

The second equality of the Lemma's statement follows from the analogous equality
$$
-\chi(W) = \brr (W) =  \tfrac 12 \sum_{u \in VW} (\deg u -2)
$$
for $W = \core(Y_1  \times  Y_2)$ and from the definition \eqref{a2}  of the numbers $N_i(A_1, \dots, A_m)$ that are used in the right hand sides of inequalities \eqref{a3}--\eqref{a4}. Indeed, for every $v \in VY_2$,  the term $\inq_V(v)^R$ of the sum
$$
\sum_{u \in VY_2}  \inq_V(u)^R
$$
is $-N(A_1(v), \ldots, A_m(v))$,  here the sets $A_j(v)$ are defined as in \eqref{Adu},   which in turn is equal to
$$
- \sum_{\al_2(w) = v}   (\deg w -2) ,
$$
where the summation runs over all vertices $w \in W = \core(Y_1  \times  Y_2) $  with $\al_2(w) = v$. Since the  map $\al_2$ is surjective, we obtain
$$
\sum_{u \in VY_2}  \inq_V(u)^R = -2 \brr ( \core(Y_1  \times  Y_2) ) ,
$$
as required.
\end{proof}

Let $A$ be a finite set. A {\em combination with repetitions} $B$ of $A$, denoted
$$
B = [[ b_1, \dots, b_\ell ]]  \sqsubseteq  A ,
$$
is a finite unordered collection of multiple copies of elements of $A$. Hence, $b_i \in A$ and
$b_i = b_j$ is possible when $i \ne j$. If $B = [[ b_1, \dots, b_\ell ]]$ is a combination with repetitions
then the cardinality $|B|$ of  $B$ is $|B| := \ell$.
\smallskip

Observe that the graph $Y_2$ of  Lemma~\ref{lem1}  can be used to construct a combination with repetitions, denoted
$$
\inq(VY_2) ,
$$
of the system   $\SLI[Y_1]$, whose elements are individual inequalities,  so that  every inequality  $q = \inq_V(u)$  of  $\SLI[Y_1]$ occurs in $\inq(VY_2)$ a number of times equal to the number of preimages of $q$ in   $VY_2$ under the map $\inq_V$. It follows from Lemma~\ref{lem1} that if
$$
\inq(VY_2)  = [[ q_1, \dots, q_\ell ]]   \sqsubseteq  \SLI[Y_1]
$$
then
$$
\sum_{q \in  \inq(VY_2)}q^L := \sum_{i=1}^{\ell}q_i^L   =  -C x_s \ ,
$$
where $C \ge 0$ is an integer, $C = 2\brr(Y_2)$.
\smallskip

For convenience of references, we introduce the following property of a graph $Y_2$ (which need not be connected).

\begin{enumerate}
\item[(B)] \   $Y_2$ is a  finite reduced $U_m$-graph such that the map   $\al_2 : \core(Y_1  \times  Y_2) \to Y_2$ is surjective,  $\core(Y_2) = Y_2$,  and  $\brr(Y_2) = - \chi(Y_2) > 0$.
\end{enumerate}

Note that the equality $\core(Y_2) = Y_2$  could be dropped as it follows from the surjectivity of the map
$\al_2 : \core(Y_1  \times  Y_2) \to Y_2$.

\begin{lemma}\label{lem2}  Suppose $Q$ is a nonempty combination with repetitions of
$\SLI[Y_1]$  and
\begin{equation}\label{b2}
\sum_{q \in  Q}^{}q^L = -C(Q) x_s  ,
\end{equation}
where $C(Q) > 0$ is an integer. Then there exists a finite reduced $U_m$-graph $Y_{2,Q}$ with property (B) such that, letting  $ \wtl Q := \inq(VY_{2,Q})$, one has  $|\wtl Q | = |Q|$ and
\begin{align}\label{b3}
\sum_{q \in  Q}^{}q^L & = \sum_{q \in \wtl Q}^{}q^L =     - 2\brr(Y_{2,Q}) x_s , \\ \label{b3b3}
 \sum_{q \in  Q}^{}q^R & \ge \sum_{q \in \wtl Q}^{}q^R  =  - 2\brr(  \core(Y_1  \times  Y_{2,Q})  )    .
\end{align}
\end{lemma}

\begin{proof} We will construct a graph $Y_{2,Q}$ whose vertices $ u_1, \dots, u_{|Q|} $ are in
bijective correspondence
$$
u_i \mapsto q_i ,  \quad  i =1, \dots, |Q|,
$$
with elements of the combination $Q = [[ q_1, \dots, q_{|Q|} ]] \sqsubseteq   \SLI[Y_1]$.
\smallskip

Recall that every inequality $q_i$ in $Q$ has one of the form \eqref{a3}--\eqref{a4}.
It follows from the assumption \eqref{b2} that all  terms $\pm x_{j, D}$ in the  sum
\begin{align}\label{fsum}
\sum_{i=1}^{|Q|}q_i^L
\end{align}
will cancel out.
Hence, there exists an involution  $\iota$ on the set of all terms $\pm x_{j, D}$   of the formal sum \eqref{fsum}
such that $\iota$ takes  every  term $\pm x_{j, D}$  of $q_{i_1}^L$ to a term $\mp x_{j, D}$ of $q_{i_2}^L$, where $i_2 \ne i_1$,   and $\iota^2 = \mbox{id}$.
\smallskip

If  $\iota$ takes the term  $-x_{j, D}$ of $q_{i_1}^L$ to  the term $ x_{j, D}$ of $q_{i_2}^L$,  then we connect the vertex $u_{i_1}$ to $u_{i_2}$ by an oriented edge in $Y_{2,Q}$ whose label is $a_j$,  $a_j \in \A$. This definition determines the local structure of the graph $Y_{2,Q}$ and, in particular, means that if $q_i \in Q$ has type \eqref{a3}, then $u_i$ is a 1-vertex.
On the other hand, if  $q_i \in Q$ has type \eqref{a4}, then $u_i$ is a 2-vertex. Furthermore, it follows from this definition that the degree of every vertex $u_i$ of  $Y_{2,Q}$  is $k(q_i) \ge 2$. Recall that
$k(q_i)$ denotes the parameter $k$ for $q_i$,  see the definitions \eqref{a3}--\eqref{a4}.
\smallskip

Looking at the coefficients of $-x_s$ in  \eqref{fsum}, we can see from \eqref{b2} and
\eqref{a3}--\eqref{a4} that
$$
C(Q) = \sum_{i = 1}^{|Q|}(k(q_i)-2)  >0 .
$$
Hence, the graph $Y_{2,Q}$ has a vertex of degree at least $3$.
\smallskip

Therefore, $Y_{2,Q}$  is a finite reduced $U_m$-graph  such that
$\core (Y_{2,Q}) = Y_{2,Q}$ and $\brr( Y_{2,Q}) >0$. Note that $Y_{2,Q}$ is not uniquely determined by $Q$ (as there are many choices to define the involution $\iota$, i.e., to do cancellations in the left hand side of \eqref{b2}).
\smallskip

Consider the graph $ \core( Y_1 \times Y_{2,Q})$ and the associated graph maps
$$
\al_1 :  \core( Y_1 \times Y_{2,Q}) \to Y_1 , \quad   \al_2 :   \core( Y_1 \times Y_{2,Q})   \to Y_{2,Q}  .
$$
It is clear from the definitions that $\al_2$ is surjective, hence,  $Y_{2,Q}$ has property (B).
\smallskip

Let $u_i$ be a vertex of $Y_{2,Q}$ and let the inequality $q_i \in Q$, corresponding to $u_i$, is constructed by means of an $\al_i$-admissible tuple $(A_{1, i}, \ldots, A_{m, i})$ so that
\begin{align*}
q_i^L & =  (-1)^{\al_i}   x_{j_{1, i}, A_{j_{1, i}, i} } + \cdots +
 (-1)^{\al_i} x_{j_{k(q_i), i}, A_{j_{k(q_i), i}} } - (k(q_i)-2)x_s  , \\
q_i^R & = - N_{\al_i}(A_{1, i}, \ldots, A_{m, i}) .
\end{align*}
It follows from the definitions that, for every possible term  $\pm x_{j, A_{j,i}}$  of $q_i^L$, where $ A_{j,i} \ne \varnothing$,     the set
$A_j(u_i)$, as defined in \eqref{Adu}, contains $A_{j,i}$ as a subset and, if $q_i^L$ has no term $\pm x_{j, A_{j,i}}$, i.e., $ A_{j,i} = \varnothing$,   then $A_j(u_i) = \varnothing$.  These remarks mean that if $\wtl Q := \inq(V Y_{2,Q})$ then $| \wtl Q | = |  Q |$ and $Y_{2,Q} = Y_{2, \wtl Q }$ for a suitable involution $\wtl \iota = \wtl \iota(\wtl Q)$.
\smallskip

Hence, if $q_i$ has the form \eqref{a3}, i.e., $\al_i = 1$,  then
$$
\inq_V( u_i )^L = - x_{j_{1,i}, A_{j_1}(u_i)} - \ldots - x_{j_{k(q_i), i},  A_{j_{k(q_i)}}(u_i) } - ({k(q_i)} -2)x_s ,
$$
where $A_{j_t,i } \subseteq A_{j_t}(u_i)$ for every $t =1, \ldots, {k(q_i)}$ and
$$
N_1(A_{1, i}, \ldots, A_{m, i} ) \le N_1(A_1(u_i), \ldots, A_m(u_i)) .
$$
\smallskip

Analogously, if $q_i$ has the form \eqref{a4}, i.e., $\al_i = 2$,  then
$$
\inq_V( u_i )^L =  x_{j_{1, i}, A_{j_1}(u_i)} + \ldots + x_{j_{k(q_i), i}, A_{j_{k(q_i)}}(u_i)} - ( {k(q_i)} -2)x_s ,
$$
where $A_{j_t,i} \subseteq A_{j_t}(u_i)$ for every $t =1, \ldots, {k(q_i)}$  and
$$
N_2(A_{1,i},  \ldots, A_{m, i} ) \le N_2(A_1(u_i), \ldots, A_m(u_i)) .
$$
\smallskip

Therefore,
\begin{equation*}
\sum_{q \in  \inq(VY_{2,Q})} q^R = \sum_{q \in \wtl Q} q^R \le   \sum_{q \in  Q} q^R    .
\end{equation*}
\smallskip

Now both the equality  \eqref{b3}  and  inequality  \eqref{b3b3} follow from Lemma~\ref{lem1}.
\end{proof}
\smallskip

We summarize Lemmas~\ref{lem1}--\ref{lem2} as follows.

\begin{lemma}\label{lem3}  The function
\begin{equation*}
\inq:  Y_2  \mapsto \inq(VY_2) = Q
\end{equation*}
from the set of finite reduced $U_m$-graphs
$Y_2$ with property (B) to the set of combinations  with repetitions $Q$ of  $\SLI[Y_1]$ with the property
 $\sum_{q \in Q}^{}q^L = -C(Q) x_s $,
where $C(Q) > 0$ is an integer, is such  that
\begin{equation*}
\sum_{q \in \inq(VY_2) }^{}q^L    =   -2\brr(Y_2) x_s  \quad   \mbox{and}   \quad
\sum_{q \in   \inq(VY_2) }^{}q^R  =   -2\brr(\core(Y_1 \times  Y_2)) .
\end{equation*}

In addition, for every $Q$ in the codomain of the function  $\inq$,  there exists a graph $Y_{2,Q}$ in the domain of $\inq$  such that, letting
  $\wtl Q := \inq(VY_{2,Q})$, one has  $|\wtl Q | = |Q|$ and
\begin{align*}
\sum_{q \in  Q}^{}q^L & = \sum_{q \in \wtl Q}^{}q^L =     - 2\brr(Y_{2,Q}) x_s , \\
 \sum_{q \in  Q}^{}q^R & \ge \sum_{q \in \wtl Q}^{}q^R  =  - 2\brr(  \core(Y_1  \times  Y_{2,Q})  )    .
\end{align*}
\end{lemma}

\begin{proof} This is straightforward  from Lemmas~\ref{lem1}--\ref{lem2}.
\end{proof}

\section{Utilizing the method of linear programming}

Let us briefly review relevant results from the theory of linear programming (LP) over the field $\mathbb Q$ of rational numbers.
Following the notation of Schrijver's monograph \cite{S86}, let $A \in \mathbb Q^{m'\times n'}$ be an $m'\times n'$-matrix, let $b \in \mathbb Q^{m'\times 1} = \mathbb Q^{m'} $ be a column vector,  let $c \in \mathbb Q^{1\times n'} $ be a row vector, $c= (c_1, \dots, c_{n'})$,
and let $x$ be a column vector consisting of variables $x_1, \dots, x_{n'}$, so $x = (x_1, \dots, x_{n'})^{\top}$, where $M^{\top}$ means the transpose of a matrix $M$. The inequality $x \ge 0$ means that $x_i \ge 0$ for every $i$.
\smallskip

A typical LP-problem asks about the maximal value of the objective linear function
$$
cx = c_1 x_1 \dots +c_{n'}x_{n'}
$$
over all $x \in \mathbb Q^{n'}$ subject to the system of linear inequalities $Ax \le b$.
This value (and often the LP-problem itself) is denoted
$$
\max\{ cx \mid Ax \le b  \} .
$$
\smallskip

We write $\max\{ cx \mid Ax \le b    \} = -\infty$ if the set $\{ cx \mid Ax \le b    \}$ is empty. We write $\max\{ cx \mid Ax \le b    \} = +\infty$ if the set $\{ cx \mid Ax \le b    \}$ is unbounded from above and say that
 $\max\{ cx \mid Ax \le b   \} $ is finite if the set $\{ cx \mid Ax \le b    \}$ is nonempty and bounded from above. The notation and terminology for an LP-problem
 $$
 \min\{ cx \mid Ax \le b    \} = -  \max\{ -cx \mid Ax \le b    \}
 $$
 is analogous with
 $-\infty$ and $+\infty$ interchanged.
\smallskip

If $\max\{ cx \mid Ax \le b   \} $ is an LP-problem defined as above, then the problem
 $$
 \min\{ b^{\top}y \mid A^{\top}y = c^{\top}, y\ge 0  \} ,
 $$
where  $y = (y_1, \dots, y_{m'})^{\top}$, is called the {\em dual} problem of  $\max\{ cx \mid Ax \le b   \}$.
\smallskip

 The (weak) duality theorem of linear programming can now be stated as follows, see \cite[Sections 7.4, 14.3]{S86}.

 \begin{TA} Let  $\max\{ cx \mid Ax \le b   \} $ be an LP-problem and let
 $$
 \min\{ b^{\top}y \mid A^{\top}y = c^{\top}, y\ge 0  \}
 $$
 be its dual LP-problem.
 Then  for every $x \in  \mathbb Q^{n'}$ such that $ Ax \le b$  and every $y \in  \mathbb Q^{m'}$ such that $A^{\top}y = c^{\top}, y\ge 0$,
 one has that  $c x = y^{\top} Ax \le  b^{\top} y$ and
 \begin{equation}\label{dt}
\max\{ cx \mid Ax \le b    \}  =  \min\{ b^{\top}y \mid A^{\top}y = c^{\top}, y\ge 0  \}
\end{equation}
provided both polyhedra  $\{ x \mid Ax \le b \}$ and  $\{  y \mid A^{\top}y = c^{\top}, y\ge 0  \}$ are not empty.  In addition,
the minimum, whenever it is finite,  is attained at a vector $ y_V$   which is a vertex of
the polyhedron $\{ y  \mid A^{\top}y = c^{\top}, y\ge 0 \}$.
\end{TA}

We now consider the problem of maximizing the objective linear function
$$
cx :=  -x_s
$$
over all rational vectors $x$, $x \in \mathbb Q^{n'}$, for a suitable $n'$, subject to the system of  linear inequalities $\SLI[Y_1]$, as an LP-problem $\max \{ cx \mid Ax \le b  \}$.
Note that, in this context, $m' = m_{\inq}$ and $n' = n_{\inq}$, where $m_{\inq}$ is the number of inequalities in $\SLI[Y_1]$
and $n_{\inq}$ is the number of all variables $x_{j, B}, x_s$ in $\SLI[Y_1]$.
\smallskip

It is straightforward to verify that the dual problem
$$
\min \{ b^{\top} y  \mid  A^{\top} y = c^{\top}, y\ge 0 \}
$$
of this  LP-problem  $\max \{ cx \mid Ax \le b  \}$  can be equivalently stated as follows
\begin{equation}\label{dlp0}
\sum_{i=1}^{m_{\inq} } y_i q_i^R \to \min  \quad  \mbox{subject to} \quad   y \ge 0 , \
\sum_{i=1}^{m_{\inq}} y_i q_i^L = - x_s .
\end{equation}
We emphasize that the last equality should be thought of as a formal linear combination on variables
$x_{j, B}, x_s$.
Note we can  write \eqref{dlp0} in the form
\begin{equation}\label{dlp}
\min \bigg\{ \sum_{i=1}^{m_{\inq}} y_i q_i^R \  \Big| \   y \ge 0 ,    \sum_{i=1}^{m_{\inq}} y_i q_i^L = - x_s  \bigg\} .
\end{equation}
\smallskip

In Lemma~\ref{lem3},  we established the existence of a function
$$
\inq :   Y_2 \mapsto  \inq(Y_2) ,
$$
from the set of finite reduced $U_m$-graphs $Y_2$ with property (B) to a certain set of combinations with  repetitions of $\SLI[Y_1]$.
Now we will relate these combinations
with repetitions of $\SLI[Y_1]$ to solutions of the dual LP-problem \eqref{dlp}.
\smallskip

Consider a  combination $Q$  with repetitions of $\SLI[Y_1]$ that has the property that
 \begin{equation}\label{st1}
\sum_{q \in Q} q^L = -C(Q) x_s     ,
\end{equation}
where $C(Q) >0$ is an integer.
As above in  \eqref{dlp0}--\eqref{dlp},  let  all of the inequalities in $\SLI [Y_1]$ be indexed and
$$
\SLI [Y_1] = \{ q_1, \dots, q_{m_{\inq} } \} .
$$
Let  ${\eta}_i(Q) \ge 0$ denote the number of times that $q_i$ occurs in $Q$, and let $\kappa_i$ be the coefficient of $x_s$ in $q_i$. Then it follows from the definitions and \eqref{st1} that
\begin{equation}\label{st2}
 \sum_{q \in Q} q^L =  \sum_{i=1 }^{m_{\inq}}  \kappa_i {\eta}_i(Q) x_s =  - C(Q) x_s    .
\end{equation}
\smallskip

Consider the map
\begin{equation}\label{dfyy}
\sol :  Q \mapsto y_Q = ( y_{Q,1}, \dots,  y_{Q, m_{\inq}} )^{\top}   ,
\end{equation}
where
$y_{Q,i} :=  \frac{{\eta}_i(Q)}{C(Q)}$ for $i =1, \dots, m_{\inq}$.
It follows from the definitions that $y_Q$ is a rational vector, $y_Q \ge 0$, and, by \eqref{st2}, $y_Q$ satisfies the condition that
$$
\sum_{i=1}^{m_{\inq}} y_{{Q,i}} q_i^L  = - x_s    .
$$
Hence, $y_Q$ is a vector in the feasible  polyhedron
\begin{equation}\label{fpdp}
\bigg\{ y \  \Big| \   y \ge 0, \sum_{i=1}^{m_{\inq}} y_{i} q_i^L   = - x_s \bigg\}
\end{equation}
of the dual  LP-problem \eqref{dlp}.
\smallskip

Note that, in place of \eqref{dfyy},  we could also write
\begin{equation}\label{solQ}
\sol :  Q \mapsto  C(Q)^{-1}  \eta(Q)^{\top} ,
\end{equation}
where  $\eta(Q) = (  \eta_1(Q) , \dots, \eta_{m_{\inq}}(Q)   )$, as  $y_Q =   C(Q)^{-1}  \eta(Q)^{\top}$.

\smallskip

Conversely, let $z = (z_1, \dots, z_{m_{\inq}})^{\top}$ be a vector of the feasible polyhedron
\eqref{fpdp}
of the dual  LP-problem \eqref{dlp}. Let $C >0 $ be a common multiple of positive denominators
of the rational numbers
$z_1, \dots, z_{m_{\inq}}$. Consider a combination with repetitions
$Q(z)$ of  $\SLI[Y_1]$ such that every $q_i$ of  $\SLI[Y_1]$ occurs
in $Q(z)$ exactly  $C z_{i} = {n}_i$ many times. Then it follows from the definitions that
\begin{equation}\label{c4}
\sum_{q \in Q(z)} q^L =   \sum_{i=1}^{m_{\inq}} {n}_i q_i^L = \sum_{i=1}^{m_{\inq}} C z_{i}  q_i^L  =
 C \sum_{i=1}^{m_{\inq}} z_{i}  q_i^L  = - C x_s  .
\end{equation}
Now we can see from
\begin{equation}\label{c4x}
\frac {\eta_i(Q(z)) }{C} = \frac {C z_i}{C} = z_i ,
\end{equation}
where $i = 1, \dots, m_{\inq}$, that the vector $y_{Q(z)} = \sol(Q(z))$,  defined by \eqref{dfyy} for $Q(z)$,
is equal to $z$.

\begin{lemma}\label{lem4}
The map
$$
\sol :  Q \mapsto \sol(Q) = y_Q ,
$$
defined by \eqref{dfyy}, is a
surjective function from the set of combinations $Q$ with repetitions of $\SLI[Y_1]$ that
satisfy the equation $\sum_{q \in Q} q^L = -C(Q) x_s$, where $ C(Q) >0$ is an integer, to the  feasible polyhedron   \eqref{fpdp}    of the dual  LP-problem
\eqref{dlp}. Furthermore, the composition of the maps
$\inq$ and $\sol$,
$$
\sol  \circ \inq :  Y_2 \mapsto  \sol( \inq(Y_2) )  = y_{Y_2} ,
$$
is a function from the set of graphs with property (B) to the  polyhedron   \eqref{fpdp}    of  \eqref{dlp}.
 Under this map, the value of the objective function
$
\sum_{i=1}^{m_{\inq}}   y_{Y_2, i}  q_i^R
$
of the  dual  LP-problem  \eqref{dlp} at $y_{Y_2}$ satisfies the equality
\begin{equation}\label{ps4}
\sum_{i=1 }^{m_{\inq}}  y_{Y_2, i}    q_i^R   =  -\frac{\brr ( \core(Y_1  \times  Y_2) )}{\brr (Y_2) }   .
\end{equation}
\smallskip

In addition, for every $z $ in the  polyhedron     \eqref{fpdp}, there is a vector
$\wtl z$   in   \eqref{fpdp}  such that $\wtl z = \sol( \inq(Y_2))$ for some graph $Y_2$ with property (B) and
\begin{equation}\label{ps44}
\sum_{i=1 }^{m_{\inq}}    \wtl z_{i}  q_i^R    \le \sum_{i=1 }^{m_{\inq}}  z_{i}  q_i^R  .
\end{equation}
\end{lemma}

\begin{proof} As was established above, see computations   \eqref{c4}--\eqref{c4x},  $\sol$ is a surjective function.
\smallskip

Consider a finite irreducible $U_m$-graph  $Y_2$ with property (B) and define
$$
Q := \inq(Y_2) , \quad  y_{Y_2} := \sol(Q) .
$$

By  Lemma~\ref{lem3}, we have
\begin{equation}\label{ps5}
\sum_{q \in Q} q^L = - 2\brr (Y_2) x_s  \quad \mbox{ and} \quad  \sum_{q \in Q} q^R = - 2\brr (\core(Y_1 \times  Y_2) ) .
\end{equation}

It follows from
\eqref{st2} and \eqref{ps5} that  $C(Q ) = 2\brr(Y_2)$. Hence,  using  the  definition  \eqref{dfyy} and equalities \eqref{ps5}, we obtain
\begin{equation*}
\sum_{i=1}^{m_{\inq}}  y_{Y_2, i}   q_i^R = \frac { \sum_{q \in Q } q^R }
{ C(Q) } = -\frac {\brr( \core(Y_1  \times  Y_2) ) }{ \brr(Y_2 ) }   ,
 \end{equation*}
as required in \eqref{ps4}.
\smallskip

To prove the additional statement, consider a vector $z$  in the polyhedron  \eqref{fpdp}.
Since  $\sol$ is  surjective, there is a combination with repetitions $Q$ such that $\sol(Q) = z$.
By Lemma~\ref{lem3} for this $Q$,
there is a graph $Y_{2,Q}$ such that, letting $\inq(VY_{2,Q}) = \wtl Q$, we have that
$|\wtl Q | = |Q|$ and
\begin{align}\label{ABA1}
\sum_{q \in  Q}^{}q^L & = \sum_{q \in \wtl Q}^{}q^L =     - 2\brr(Y_{2,Q}) x_s = - C(Q)x_s = - C(\wtl  Q)x_s  ,  \\ \label{ABA2}
\sum_{q \in  Q}^{}q^R & \ge \sum_{q \in \wtl Q}^{}q^R  =  - 2\brr(  \core(Y_1  \times  Y_{2,Q})  )  .
\end{align}
\smallskip

Let $\wtl z := \sol( \wtl Q)$. Then, in view of \eqref{ABA1}--\eqref{ABA2}, we obtain
\begin{equation*}
\sum_{i=1}^{m_{\inq}}  \wtl z_i q_i^R = \frac { \sum_{q \in \wtl Q } q^R }
{ C(\wtl Q) } \le \frac { \sum_{q \in  Q } q^R }
{ C(Q) } =   \sum_{i=1}^{m_{\inq}}   z_i q_i^R   ,
\end{equation*}
as required.
\end{proof}

We will say that a real nonnegative number $\sigma(Y_1)$ is the WN{\em-coefficient} for the graph $Y_1$ if
\begin{equation*}
\brr( \core(Y_1  \times  Y_2)   ) \le \sigma(Y_1) \brr(Y_1) \brr(Y_2)
\end{equation*}
for every finite reduced $U_m$-graph $Y_2$ and  $\sigma(Y_1)$ is minimal with this property.

\begin{lemma}\label{lem5}  The WN-coefficient  $\sigma(Y_1)$ for the graph $Y_1$ is equal to
\begin{equation}\label{sup}
\sup_{Y_2} \bigg\{  \frac {\brr( \core(Y_1  \times  Y_2)  )}{\brr(Y_1) \brr(Y_2)}  \bigg\}
\end{equation}
over all graphs $Y_2$ with property~(B).
\end{lemma}

\begin{proof} Since $\brr(Y_1) >0$ and $\core(Y_1) = Y_1$, we may use $Y_2 = Y_1$ to see that
$$
\sigma(Y_1) = \sup_{Y_2} \frac {\brr( \core(Y_1  \times  Y_2) )}{\brr(Y_1)  \brr(Y_2)}  \ge
 \frac {\brr( \core(Y_1  \times  Y_1) )}{\brr(Y_1)^2}  \ge
 \frac {1}{\brr(Y_1) }    >0
$$
over all finite reduced  $U_m$-graph $Y_2$ such that $\brr(Y_2) >0$. We may also assume that
$Y_2$ has no vertices of degree less than 2, i.e., $\core(Y_2) = Y_2$.
\smallskip

Suppose that a graph $Y_2$ is such that  $\brr(Y_2) >0$, $\core(Y_2) = Y_2$
and $Y_2$ does not satisfy property (B). This means that the projection
$$
\al_2 :  \core(Y_1  \times  Y_2)   \to Y_2
$$
is not surjective. We delete
those edges and vertices in $Y_2$ that have no preimages in $\core(Y_1  \times  Y_2)$ under $\al_2$. As a result, we obtain a subgraph
$Y_2'$ of $Y_2$ such that $\core(Y_1  \times  Y'_2) = \core(Y_1   \times  Y_2)$ and
$0 <  \brr(Y'_2)< \brr(Y_2)$.  Since
$$
\frac{\brr( \core(Y_1  \times  Y'_2) )}{\brr(Y_1)  \brr(Y'_2)  }  >
\frac{\brr( \core(Y_1  \times  Y_2) ) }{\brr(Y_1)  \brr(Y_2) } ,
$$
it follows  that the graphs $Y_2$ that do not satisfy property (B) can be disregarded when taking the
supremum \eqref{sup}.
\end{proof}

\begin{lemma}\label{lem6}
Both optima $\max \{ - x_s \mid \SLI[Y_1]\}$ and
$$
\min \bigg\{  \sum_{j=1}^{m_{\inq}}   y_j q_j^R  \  \Big|  \   y \ge 0 , \
\sum_{j=1}^{m_{\inq}}   y_j q_j^L  = - x_s  \bigg\}
$$
are finite and satisfy the following equalities
\begin{align}
\label{st3}
\begin{split}
 \max \{ - x_s \mid \SLI [Y_1]\} & = \min \bigg\{ \sum_{i=1}^{ m_{\inq} }
 y_i q_i^R   \ \Big| \  y \ge 0 ,  \sum_{i=1}^{ m_{\inq} }   y_i q_i^L  = - x_s \bigg\} \\
  &  =   - \sigma(Y_1) \brr(Y_1)   .
\end{split}
\end{align}
Furthermore, the minimum is attained at a vector $\wtl y_V$ of the polyhedron
\eqref{fpdp} of the dual LP-problem \eqref{dlp} such that
 there is a graph $Y_{2, Q_V}$ with $  \sol( \inq(Y_{2, Q_V})) = \wtl y_V$.

In addition,  $\frac 1{m-2} \le \sigma(Y_1) \le 1$.
\end{lemma}

\begin{proof} Setting $Y_2 := Y_1$, we obtain a graph $Y_2$ with property (B).
Hence, by  Lemma~\ref{lem4}, $y_{Y_2} =  \sol ( \inq  (Y_2))$ is a solution  to the system
$$
y \ge 0 , \quad \sum_{j=1}^{m_{\inq} }   y_j q_j^L  = - x_s ,
$$
so the feasible polyhedron \eqref{fpdp} of the dual LP-problem  \eqref{dlp}  is not empty.
\smallskip

To establish that the polyhedron   $ \{ x \mid \SLI[Y_1]\}$  is not empty either,  we will show that the vector $\wht x$, whose components are
 $\wht x_{j, B} := 0$ for all $j=1,\dots,m$ and for all  $B \subseteq E_{a_j} Y_1$,  and $\wht  x_{s} := 2 \brr(Y_1)$, is a solution to $\SLI[Y_1]$.
 To do this,  we need to check that every inequality \eqref{a3}--\eqref{a4} of $\SLI[Y_1]$ is satisfied with these values of variables, that is,
 \begin{align}\label{dd5}
 -(k-2)\cdot 2 \brr(Y_1) \le -   N_i(A_1, \dots, A_m)
\end{align}
 for every $i$-admissible tuple $(A_1, \dots, A_m)$ in which exactly $k$ sets, among  $A_1, \dots, A_m$,  are nonempty.
\smallskip

It follows from the definition \eqref{a2} of $N_i(A_1, \dots, A_m)$ that $N_i(A_1, \dots, A_m) =0$ if $k=2$. Hence, if $k=2$ then
the inequality \eqref{dd5} is true.  Since $k \ge 2$, we may assume that $k>2$. Then, according to \eqref{a3a},
$N_i(A_1, \dots, A_m) \le  2 \brr(Y_1)$ and the inequality \eqref{dd5} is true again.
\smallskip

Hence, both polyhedra
$$
\{ x \mid \SLI[Y_1]\} ,  \qquad  \bigg\{  y   \  \Big| \   y \ge 0 , \sum_{i=1}^{ m_{\inq} }  y_i q_i^L  = - x_s  \bigg\}
$$
are not empty, as required.
\smallskip

According to Theorem~A, the maximum and minimum in \eqref{st3}  are finite and equal.
Referring to Theorem~A again, we obtain that the minimum  in \eqref{st3}
is attained at a vertex  $y_V$  of the polyhedron \eqref{fpdp}.
\smallskip

It follows from  Lemma~\ref{lem4} that, for the vertex  $y_V$,  there is a vector $\wtl y_V$ in the polyhedron
 \eqref{fpdp}   such that
$$
\sum_{i=1 }^{m_{\inq}}    \wtl y_{V, i}  q_i^R    \le \sum_{i=1 }^{m_{\inq}}  y_{V, i}  q_i^R
$$
and $\wtl y_V =  \sol( \inq(  Y_{2,Q_V}))$ for some graph $Y_{2,Q_V}$ with property (B).  Hence, the minimum in \eqref{st3}
 is also attained at $\wtl y_V$.
\smallskip

By Lemma~\ref{lem4}, for every graph $Y_2$ with property
(B),  the ratio $- \dfrac {\brr(\core(Y_1 \times Y_2) ) }{ \brr(Y_2 ) }$
is equal to
$$
\sum_{i=1}^{ m_{\inq} }  y_{Y_2,i} q_i^R ,
$$
where $y_{Y_2} = \sol(\inq(Y_2))$.
\smallskip

On the other hand, it follows from  Lemma~\ref{lem5} that
\begin{equation*}
\sigma(Y_1 )  \brr(Y_1 ) =  \sup_{Y_2} \bigg\{  \frac {\brr(  \core(Y_1 \times Y_2) ) }{ \brr(Y_2 ) } \bigg\} =
- \inf_{Y_2} \bigg\{  \frac {- \brr(  \core(Y_1 \times Y_2)  ) }{ \brr(Y_2 ) } \bigg\}
\end{equation*}
over all graphs $Y_2$ with property (B).
\smallskip

Therefore, putting together these facts, we obtain
\begin{align*}
 - \sigma(Y_1 )   \brr(Y_1 ) & = \inf_{Y_2} \bigg\{  \frac {- \brr(  \core(Y_1 \times Y_2)  ) }{ \brr(Y_2 ) } \bigg\} \\
  & =  \inf_{Y_2}  \bigg\{  \sum_{i=1}^{  m_{\inq} }   y_i q_i^R  \  \Big|  \    y \ge 0 , \sum_{i=1}^{  m_{\inq} }
 y_i q_i^L = - x_s \bigg\}   \\
  & =  \min \bigg\{  \sum_{i=1}^{ m_{\inq} }   y_i q_i^R     \  \Big|  \        y \ge 0 , \  \sum_{i=1}^{ m_{\inq} }
 y_i q_i^L = - x_s \bigg\}\\ &
  =     \sum_{i=1}^{ m_{\inq} }  \wtl y_{V,i} q_i^R =  \max \{ - x_s \mid \SLI[Y_1]\}   ,
\end{align*}
as desired in \eqref{st3}.
The  equalities  \eqref{st3} are proven.
\smallskip

It remains to show that $\frac 1{m-2} \le \sigma(Y_1) \le 1$.
The inequality $\sigma(Y_1) \le 1$ follows from the fact that the strengthened Hanna Neumann conjecture is true, see
\cite{Fr}, \cite{Mn}, \cite{Dp}, \cite{I15a}.
\smallskip

Let $Y_2 := U_m$. Then
$$
\brr(Y_2 ) = m-2 \quad  \mbox{and} \quad  \brr(  \core(Y_1 \times Y_2)  )  = \brr(Y_1 )
$$
because  $Y_1 \times Y_2 = Y_1$. Hence,
$$
\sigma(Y_1 ) \ge \frac{\brr(  \core(Y_1 \times Y_2)  )  }{\brr(Y_1 ) \brr(Y_2 )} = \frac 1{m-2} ,
$$
as required.
\end{proof}

\begin{lemma}\label{lem7}
There exists a finite reduced $U_m$-graph $Y_{2,Q_V}$ such that
$Y_{2,Q_V}$ has property (B),
$$
\brr(  \core(Y_1 \times Y_{2,Q_V})  )   =   \sigma(Y_1 ) \brr(Y_1 ) \brr(Y_{2,Q_V} ) ,
$$
$Y_{2,Q_V}$ is connected, and
$$
| E Y_{2,Q_V} |  < 2^{2^{  |E Y_1|/2 + 2\log_2m}} .
$$
\end{lemma}

\begin{proof}  According to Lemma~\ref{lem6} and to Theorem~A,  we may assume that
the minimum of the dual LP-problem \eqref{dlp}
is attained at a vertex $y_V$,  $y_V \ge 0$,  of  the feasible polyhedron \eqref{fpdp} of the LP-problem \eqref{dlp}.
\smallskip

Since $y_V$ is a vertex solution of the LP-problem   \eqref{dlp} and \eqref{dlp}  is stated in the form
$\min \{ b^{\top} y \mid A^{\top} y = c^{\top}, y \ge 0 \}$, where $Ax \le b$ is a matrix form  of the system
\eqref{SLI}, it follows that the vertex solution $y_V$ will satisfy $m_{\inq}$  equalities among
$$
A^{\top} y = c^{\top} ,  \quad y_i = 0 ,   \quad  i = 1, \ldots, m_{\inq}  ,
$$
whose left hand side parts  are linearly independent
(as formal linear combinations in variables  $y_1, \ldots, y_{m_{\inq}}$).
We will call these $m_{\inq}$  inequalities  {\em distinguished}.
\smallskip

The foregoing observation implies that there are $r$, $r \le m_{\inq}$, distinguished equalities in the system
$A^{\top} y = c^{\top}$ such that the submatrix $A_r^{\top}$ of $A^{\top}$,
consisting of the rows of $A^{\top}$ that correspond to the $r$ distinguished equalities, has the following property.
The rank of $A_r^{\top}$ is $r$ and deletion of the columns  of $A_r^{\top}$, that correspond to the variables $y_i$ that in turn correspond to
the distinguished equalities $y_i = 0$,  produces an $r\times r$  matrix  $A_{r \times r}^{\top}$ with
$\det A_{r \times r}^{\top} \ne 0$. Reordering the equalities in the system $A^{\top} y = c^{\top}$ and variables $y_i$  if necessary, we may assume that $A_{r }^{\top}$ consists of the first $r$ rows of $A^{\top}$ and
$A_{r \times r}^{\top}$ is an upper left submatrix of $A^{\top}$.
\smallskip

Let
$$
\bar y_V = (y_{V,1}, \dots,  y_{V,r})
$$
be the truncated version of $y_V$ consisting of the first $r$ components. It follows from the definitions that
$\bar y_V$ contains all nonzero components of $y_V$ and
$$
A_{r \times r}^{\top}  \bar y_V  =  \bar c^{\top} = (c_1, \dots,  c_r)^{\top} .
$$

Since $\sum_{i=1}^{m_{\inq}} y_{V,i} q_i^L  = - x_s$, it follows that $c_{j} =0$ if $c_{j}$ corresponds to  a variable $x_{\ell, B}$ and $c_j =-1$ if $c_j$ corresponds to the variable  $x_{s}$. Since $y_V  \ne 0$ following from the definition of the LP-problem \eqref{dlp},
we conclude that $\bar c^{\top} \ne 0$,  i.e.,  one of $c_j$ is $-1$ and all other entries in  $\bar c^{\top}$ are equal to $0$.
Since every entry of $A_{r \times r}$ is 0 or $\pm 1$ or  $-(k-2)$, where $2 \le k \le m$, and  since  every row
of $A_{r \times r} $ contains at most
$m+1$ nonzero entries, at most one of which is different from $\pm 1$, see the definitions \eqref{a3}--\eqref{a4},
it follows that the standard Euclidian norm  of any row of $ A_{r \times r} $ is at most
$$
(m + (m-2)^2)^{1/2} < m
$$
as $m \ge 3$. Hence, by the Hadamard's inequality, we obtain
\begin{equation}\label{cr1}
| \det A_{r \times r} | <    m^{r}  .
\end{equation}

Invoking the Cramer's rule, we further get that
\begin{equation}\label{cr2}
y_{V,i} = \frac{\det A_{r \times r, i}^{\top}(\bar c^{\top})}{  \det A_{r \times r}}   ,
\end{equation}
where  $ A_{r \times r, i}^{\top}(\bar c^{\top})$ is the matrix obtained from $A_{r \times r}^{\top}$ by replacing the $i$th column with $\bar c^{\top}$, $i=1, \ldots, r$.  Similarly to \eqref{cr1}, we have that
\begin{equation}\label{cr3}
| \det A_{r \times r, i}(\bar c^{\top}) | <  m^{r}  .
\end{equation}

In view of \eqref{cr1}--\eqref{cr3},  we can see that there is a common denominator
$C >0$ for the rational numbers $y_{V,1}, \dots,  y_{V,r}$  that satisfies $C < m^{r}$ and that the nonnegative integers
$C y_{V,1}, \ldots,  C y_{V,r}$  are less than $m^{r}$.
\smallskip

It follows from the definition  \eqref{dfyy} of the function $\sol$, see also Lemma~\ref{lem4} and  \eqref{c4}--\eqref{c4x}, that if $Q_V$ is a combination such that $y_{V} = \sol(Q_V)$ and $|Q_V|$ is minimal with this property, i.e.,  the entries of $\eta(Q_V)$  are coprime,  then
\begin{equation}\label{cr3a}
| Q_V | < r m^{r} .
\end{equation}
Recall that the cardinality $| Q |$ of a combination  $Q$ is
defined so that every $q \in Q$ is counted as many times as it occurs in $Q$.

\smallskip

We now construct a graph $Y_{2, Q_V}$ from $Q_V$ as described in the proof of Lemma~\ref{lem2}.
\smallskip

It follows from  the definitions and Lemmas~\ref{lem4},~\ref{lem6}
that if
$$
\wtl y_V :=  \sol(\inq(V Y_{2,Q_V}))
$$
then the minimum of the dual LP-problem \eqref{dlp}
is also attained at $\wtl y_V$ and this minimum is equal to $-\sigma(Y_1) \brr(Y_1)$. Hence,
\begin{equation*}
\brr(\core( Y_1 \times Y_{2,Q_V})) = \sigma(Y_1) \brr(Y_1) \brr(Y_{2,Q_V}) .
\end{equation*}

Since $| V  Y_{2,Q_V} | = |Q_V|$, we have from \eqref{cr3a} that
 \begin{equation}\label{cr4}
| E Y_{2,Q_V} | \le  m | V Y_{2,Q_V} | = m | Q_V |  <  r m^{r+1}  .
\end{equation}

Note that $r$ does not exceed  the total number $n_{\inq}$ of variables $x_{j, B}, x_s$ of $\SLI[Y_1]$.
It is clear that $| E_{a_j} Y_1 | \le | E Y_1 |/2$ for every $j =1, \dots, m$.
Hence, the number of variables $x_{j, B}$ for a fixed $j$ is at most $2^{ | E Y_1 |/2 }-1$.  Therefore,
\begin{equation}\label{cr5}
r  \le n_{\inq}  \le m ( 2^{ | E Y_1 |/2 }-1) + 1 \le m \cdot 2^{ | E Y_1 |/2 } - 2 .
\end{equation}

Finally, we obtain from \eqref{cr4}--\eqref{cr5} that
 \begin{align}
 \begin{split}\label{cr5a}
| E Y_{2,Q_V} | <  r m^{r+1}  &  \le  ( m \cdot 2^{ | E Y_1 |/2 } -2)  m^{  m \cdot 2^{  | E Y_1 |/2 }  -1   } \\
 & <   2^{ | E Y_1| /2  }  m^{  m \cdot 2^{ | E Y_1 |/2 }  } \\
 & \le  2^{ | E Y_1| /2+ (\log_2 m) \cdot m \cdot 2^{| E Y_1 |/2}   } \\
 & <  2^{ (1+  m \cdot \log_2 m ) \cdot 2^{| E Y_1 |/2}   } \\ & <  2^{ m^2  \cdot 2^{| E Y_1 |/2}    } \\
 & \le  2^{ 2^{| E Y_1 |/2  +2\log_2 m  } }   .
 \end{split}
 \end{align}

It remains to show that the graph $Y_{2,Q_V}$ is connected.
\smallskip

Arguing on the contrary, assume that  the graph   $Y_{2,Q_V}$   is the disjoint union of its two subgraphs   $Y_3$ and $ Y_4$. First we assume that
\begin{equation}\label{case1}
\brr(Y_3) >0  \quad   \text{and}  \quad   \brr(Y_4) >0 .
\end{equation}

Clearly, $Y_3$ and $Y_4$ are  graphs with property (B).
Recall that the vertices of the graph   $Y_{2,Q_V}$ bijectively correspond to the inequalities of the combination $Q_V$,
see the proof of Lemma~\ref{lem2}.  In particular, we can consider  the combinations $Q_{3}$ and $Q_{4}$, whose inequalities bijectively correspond to the vertices of $Y_3$ and $Y_4$, resp. It is clear that $Q_V$ is the union of combinations $Q_{3}$ and $Q_{4}$ and
\begin{equation}\label{etaQ}
\eta(Q_V) = \eta( Q_{3} ) +  \eta(  Q_{4}) .
\end{equation}
We specify that the union $B_1  \sqcup B_2$  of two combinations  $B_1, B_2$ is the combination whose elements are all elements  of both $B_1  $ and $ B_2$, in particular, $|B_1  \sqcup B_2| =  |B_1| + |B_2|$.
\smallskip

Furthermore, the graphs  $Y_3$ and $Y_4$ could be
constructed from $Q_{3}$ and $Q_{4}$, resp., in  the same manner as $Y_{2,Q_V}$ was constructed from $Q_{V}$.
In particular,  the combinations  $Q_{3}$ and $Q_{4}$ belong to the domain of  the function $\sol$.
\smallskip

Invoking  Lemma~\ref{lem4}, denote $y_V(j)  := \sol( Q_{j} )$,  $j=3,4$.
We also denote
$$
\sum_{q \in Q_V}q^L =  - C(Q_V) x_s ,  \quad  \sum_{q \in Q_{j} } q^L =  - C(Q_{j}) x_s ,
$$
where $j = 3,4$.
\smallskip

Since $ Q_V =  Q_{3} \sqcup Q_{4}$, it follows that  $ C(Q_V) =  C(Q_{3}) + C(Q_{4})$.
According to the definition  \eqref{dfyy} of the function $\sol$, we have
\begin{equation}\label{etayX}
y_{V,i}  = \frac{\eta_i(Q_{V} ) }{ C(Q_{V}) } ,   \qquad   y_{V,i}(j)  = \frac{\eta_i(Q_{j}) }{ C(Q_{j})  }
\end{equation}
for all suitable $i, j$.
Hence, in view of  \eqref{etaQ}, for every $i = 1, \dots, m_{\inq}$, we obtain
\begin{align}\label{eqQC}
\begin{split}
y_{V,i} & = \frac{\eta_i(Q_{V} ) }{ C(Q_{V}) }   = \frac{\eta_i(Q_{3} )  +\eta_i(Q_{4} ) } { C(Q_{V}) }  \\
 & = \frac{C ( Q_{3} )  } { C(Q_{V}) } \cdot  \frac{\eta_i(Q_{3} )  } { C(Q_{3}) }      +
     \frac{C ( Q_{4} )  } { C(Q_{V}) }  \cdot   \frac{ \eta_i(Q_{4} ) } { C(Q_{4}) } \\
 & =  \lambda_3  y_{V, i}(3) +  \lambda_4  y_{V, i}(4)  ,
\end{split}
\end{align}
where $\lambda_3 =  \frac{C ( Q_{3} )  } { C(Q_{V}) } $ and  $\lambda_4 =  \frac{C ( Q_{4} )  } { C(Q_{V}) } $ are positive rational numbers that satisfy $\lambda_3+ \lambda_4 =1$.
\smallskip

The equalities  \eqref{eqQC} imply that
\begin{equation}\label{ylam}
y_V = \lambda_3  y_V(3) +  \lambda_4  y_V(4)  .
\end{equation}

Since $y_V$ is a vertex of the polyhedron  \eqref{fpdp},  $y_V(3)$ and $ y_V(4) $ are vectors in  \eqref{fpdp},    and $0 < \lambda_3, \lambda_4 < 1$,  $\lambda_3+ \lambda_4 =1$, it follows  from  \eqref{ylam}  that
$$
y_V(3) =  y_V(4) =  y_V .
$$
Hence, in view of \eqref{etayX}, the tuples $\eta(Q_V)$, $\eta( Q_{3} )$, $ \eta(  Q_{4})$  that have integer entries  are rational multiples of each other.  Referring to \eqref{etaQ}, we conclude that the entries  of $\eta(Q_V)$  are not coprime, contrary to the definition of the combination $Q_V$.  This contradiction completes the case  \eqref{case1}.
\medskip

Assume that  the graph   $Y_{2,Q_V}$   is the disjoint union of its two subgraphs   $Y_3$ and $ Y_4$ such that
\begin{equation}\label{case2}
\brr(Y_3) >0  \quad   \text{and}  \quad   \brr(Y_4) = 0 .
\end{equation}
\smallskip

Let $2Q_V$ denote  the combination such that  $\eta(2Q_V) = 2 \eta(Q_V)$, i.e., to get $2Q_V$  from $Q_V$ we double the number of occurrences of each inequality in $Q_V$. Using this combination $2Q_V$, we can construct, as in the proof of Lemma~\ref{lem2}, a  graph $Y_{2, 2Q_V}$ which consists of two disjoint copies of $Y_{2,Q_V}$, denoted $\bar Y_{2,Q_V}$ and $\wht  Y_{2,Q_V}$. Since  $Y_{2,Q_V} = Y_3 \cup Y_4$, we can represent
the graph $Y_{2, 2Q_V}$ in the form
$$
Y_{2, 2Q_V} = Y_5    \cup Y_6,
$$
 where $ Y_5 :=  \bar Y_3  \cup \bar Y_4  \cup \wht Y_4$ and  $ Y_6 :=  \wht Y_3$
 \smallskip

Clearly, $\brr(Y_5) >0 $,  $\brr(Y_6) >0 $,   and  both  $Y_5,  Y_6$ have  property (B).
As above, we remark that the vertices of  $Y_{2, 2Q_V}$ are in bijective correspondence with the inequalities of  $2Q_V$. Hence,   the combination $2Q_V$ is the  union of the combinations $Q_{5}$ and $Q_{6}$ that consist of those inequalities that correspond to the vertices of   $Y_5$ and $Y_6$, resp., and that can be used to construct the graphs   $Y_5$ and $Y_6$ in the same manner as $Y_{2,Q_V}$ was constructed from $Q_V$.
\smallskip

As above, we can write
\begin{equation}\label{etaQa}
\eta(2Q_V) =    \eta( Q_{5} ) +  \eta(  Q_{6}) .
\end{equation}

Note that  the combinations  $Q_{5}$ and $Q_{6}$ belong to the domain of  the function $\sol$.
Using  Lemma~\ref{lem4}, denote $y_V(j)  := \sol( Q_{j} )$,  $j=5,6$.
As above, denote
$$
\sum_{q \in 2Q_V}q^L =  - C(2Q_V) x_s ,  \quad  \sum_{q \in Q_{j} } q^L =  - C(Q_{j}) x_s ,
$$
where $j = 5,6$.
\smallskip

Since $ 2Q_V =  Q_{5} \sqcup Q_{6}$, it follows that  $ C(2Q_V) =  C(Q_{5}) + C(Q_{6})$.
According to the definition  \eqref{dfyy} of the function $\sol$, we have
\begin{equation}\label{etaya}
y_{V,i}  = \frac{\eta_i(Q_{V} ) }{ C(Q_{V}) } = \frac{\eta_i(2Q_{V} ) }{ C(2Q_{V}) } ,  \qquad   y_{V,i}(j)  = \frac{\eta_i(Q_{j}) }{ C(Q_{j})  }
\end{equation}
for all suitable $i, j$.
Hence, in view of  \eqref{etaQa}, for every $i = 1, \dots, m_{\inq}$, we obtain
\begin{align}\label{eqQCa}
\begin{split}
y_{V,i} & = \frac{\eta_i(2Q_{V} ) }{ C(2Q_{V}) }   = \frac{\eta_i(Q_{5} )  +\eta_i(Q_{6} ) } { C(2Q_{V}) }  \\
 & = \frac{C ( Q_{5} )  } { C(2Q_{V}) } \cdot  \frac{\eta_i(Q_{5} )  } { C(Q_{5}) }      +
     \frac{C ( Q_{6} )  } { C(2Q_{V}) }  \cdot   \frac{ \eta_i(Q_{6} ) } { C(Q_{6}) } \\
 & =  \lambda_5  y_{V, i}(5) +  \lambda_6  y_{V, i}(6)  ,
\end{split}
\end{align}
where $\lambda_5 =  \frac{C ( Q_{5} )  } { C(2Q_{V}) } $ and  $\lambda_6 =  \frac{C ( Q_{6} )  } { C(2Q_{V}) }$ are positive rational numbers that satisfy $\lambda_5+ \lambda_6 =1$.
\smallskip

The equalities  \eqref{eqQCa} imply that
\begin{equation}\label{ylama}
y_V = \lambda_5  y_V(5) +  \lambda_6  y_V(6)  .
\end{equation}

Since $y_V$ is a vertex of the polyhedron  \eqref{fpdp}, $y_V(5)$ and $ y_V(6) $ are vectors in  \eqref{fpdp},     and $0 < \lambda_5, \lambda_6 < 1$,  $\lambda_5+ \lambda_6 =1$, it follows  from  \eqref{ylama}  that
$$
y_V(5) =  y_V(6) =  y_V.
$$
Hence, in view of \eqref{etaya}, the tuples $\eta(2Q_V)$, $\eta( Q_{5} )$, $ \eta(  Q_{6})$  that have integer entries   are rational multiples of each other.
Referring to \eqref{etaQa} and keeping in mind that   the  entries  of $\eta(Q_V) $  are  coprime, we conclude that  \begin{equation}\label{QVe}
\eta(Q_V) = \eta( Q_{5} ) = \eta(  Q_{6}) ,
\end{equation}
i.e., $Q_V = Q_{5} = Q_{6}$.   However, $Y_6 = \wht Y_3$ and $\wht  Y_3$ is a subgraph of $\wht  Y_{2,Q_V}$ that consists of several connected components of $\wht  Y_{2,Q_V}$ and  $Y_3 \ne Y_{2,Q_V}$. Hence, $ Q_{5} \ne  Q_V $.
This contradiction to \eqref{QVe} completes the second case \eqref{case2}. Thus  the graph
 $Y_{2,Q_V}$ is connected and
Lemma~\ref{lem7} is proven.   \end{proof}

\section{Proofs of Theorems~\ref{th1} and  \ref{pr1}}

\begin{proof}[Proof of Theorem~\ref{th1}]   (a)  Suppose that  $H_1$ is a finitely generated noncyclic subgroup of the free group $F_{U_m} = \pi_1(U_m, o_1)$ of rank $m -1 \ge 2$.
 Conjugating $H_1$ if necessary, we may assume that the reduced $U_m$-graph of $H_1$, denoted as above by $Y_1$, coincides with its core,
 $\core(Y_1) =  Y_1$.
\smallskip

As in Section~3, see  \eqref{SLI},   consider a system of linear inequalities $\SLI[Y_1]$ with integer coefficients associated with the graph $Y_1 = Y_1(H_1)$ and the LP-problem
\begin{equation}\label{lppf}
    \max \{ - x_s \mid \SLI[Y_1] \}  .
\end{equation}
According to Theorem~A and Lemma~\ref{lem6}, the maximum of the  LP-problem
\eqref{lppf} is equal to $-\sigma(Y_1) \brr(Y_1)$, as required.
\smallskip

(b) This part follows from the definitions and  Lemmas~\ref{lem5}--\ref{lem7} applied to the graph $Y_1$ of $H_1$.
\smallskip

(c)  It follows from the definitions of Section~3 that we can algorithmically write down  the system   $\SLI[Y_1]$ and this can be done in exponential time {}in the size of $Y_1$. Note that the size of  the graph $Y_1$ is polynomial in the size  of the input (which is a generating set for $H_1$ or the graph $Y_1$ itself). It follows from the bound \eqref{cr5} and from the definitions that the number $n_{\inq}$ of variables of
$ \SLI[Y_1]$ and the number $m_{\inq}$ of inequalities of $ \SLI[Y_1]$ satisfy
\begin{align*}
n_{\inq} & \le m( 2^{|E Y_1  |/2} -1) +1 < m 2^{|E Y_1  |/2} , \\
m_{\inq} & < 2\cdot 2^{ |E_{a_1} Y_1 | + \cdots + |E_{a_m} Y_1 |} = 2^{ |E Y_1 |/2  + 1} .
\end{align*}

Furthermore,  every number in  $ \SLI[Y_1]$ is an integer whose absolute value  is bounded by
$\max(m-2, 2 \brr(Y_1))$, see \eqref{a3a}.
Hence, the size of the primal LP-problem $\max \{ - x_s \mid \SLI[Y_1] \}$ as well as the size of the dual problem \eqref{dlp}  are exponential {}in the size of the input.
By  Lemma~\ref{lem6}, the optimal solution to the  dual problem \eqref{dlp}  is equal to
$$
 - \sigma( Y_1)  \brr (Y_1) = - \sigma( H_1)  \brr (H_1) .
$$
Since an LP-problem $\max \{cx \mid Ax \le b \} $ can be solved in deterministic polynomial time {}in the size of the problem, see \cite{S86}, and since the reduced rank $\brr(Y_1) = \brr(H_1)$ can be computed in polynomial time in the size of the input, it follows that
the WN-coefficient $\sigma(H_1)$ of $H_1$ can be computed  in exponential time  ({}in the size of the input).
\smallskip

We  recall  again that the size of the dual LP-problem \eqref{dlp}, similarly to the size of
the primal LP-problem $\max \{ - x_s \mid \SLI[Y_1] \} $, is exponential .
Next, a vertex solution  $y_V$  to  the LP-problem \eqref{dlp} can be computed in polynomial time in the
size of  \eqref{dlp}, see \cite{S86}. Note that here and below we use the notation of the proof of Lemma~\ref{lem7}.  Hence, a vertex solution $y_V$  to \eqref{dlp}  can be computed in exponential time {}in the size of $Y_1$. Using the function $ \sol$, we can compute    a  combination with repetitions $Q_V$,  such that
$\sol( Q_V) =   y_V$ and entries of  $Q_V$ are coprime, in polynomial
time {}in the size of $y_V$. The size of the vertex $y_V$, as was established in the proof of Lemma~\ref{lem7}, see
\eqref{cr1}--\eqref{cr3}, \eqref{cr5},  is  exponential. Hence, the  combination $Q_V$  can also be  computed in
 exponential time.
\smallskip

In view of inequalities \eqref{cr3a} and \eqref{cr5a}, we obtain that
\begin{align}\label{cr6}
|  Q_V| < rm^{r} <  2^{ 2^{| E Y_1 |/2  +2\log_2 m  } }  .
\end{align}

This bound, in particular, means that  every inequality $q \in \SLI[Y_1]$ occurs in $Q_V$ at most $2^{ 2^{| E Y_1 |/2  +2\log_2 m  } }$ times, hence, the number ${n}_{Q_V}(q)$ of occurrences of $q$ in $Q_V$ can be written by using
at most $2^{| E Y_1 |/2  +2\log_2 m  } $ bits.
\smallskip

As in the proofs of Lemmas~\ref{lem2}, \ref{lem7}, we construct a graph $Y_{2, Q_V}$ whose vertices are in bijective correspondence with inequalities of $Q_V$ and whose edges are defined by means of an involution $\iota_V$ on the set of terms $\pm x_{j, D}$ of left hand sides $q^L$ of the inequalities $q \in Q_V$.

\begin{lemma}\label{lem8}
The graph $Y_{2, Q_V}$ can be constructed in  deterministic exponential time {}in the size of $Y_1$.
\end{lemma}

\begin{proof} We need to explain how to compute the involution $\iota_V$ as above in exponential time ({}in the size of $Y_1$).
To do this, for each variable $x_{j, D}$ of the system $\SLI[Y_1]$, see  \eqref{SLI}, we consider a  graph $\Lambda_{j, D}$ whose set of vertices is the subset
 $$
 R_V := \{ q \mid q \in Q_V \}
 $$
 of $\SLI[Y_1]$ formed with the inequalities of $Q_V$. If $q_1, q_2 \in R_V$ are distinct, $q_1^L$ contains a term $\pm x_{j, D}$ and $q_2^L$ contains a term $\mp x_{j, D}$, resp., then we draw an edge $e$  in $\Lambda_{j, D}$ that connects $q_1$ and $q_2$. In other words, if there is a potential cancellation between terms
 $\pm x_{j, D}$  in the sum $q_1^L + q_2^L$ then  $\Lambda_{j, D}$  contains an edge
 that connects $q_1$ and $q_2$.
\smallskip

 It is clear that $\Lambda_{j, D}$  is a bipartite graph so that every edge connects a vertex of type \eqref{a3} and a vertex of type \eqref{a4}.
\smallskip

Consider a weight function
\begin{align}\label{wjD}
\om_{j,D} : E  \Lambda_{j, D} \to \mathbb Z ,
\end{align}
where $\mathbb Z$ is the set of integers, such that
$\om_{j,D}(e^{-1}) = \om_{j,D}(e) \ge 0$ and
$$
\sum_{e_- = q} \om_{j,D}(e) = {n}_{Q_V}(q) ,
$$
where  ${n}_{Q_V}(q)$ is  the number of occurrences of $q$ in $Q_V$.
\smallskip

Our nearest goal is to show that such a weight function $\om_{j,D}$  can be computed in exponential time for every pair of indices $j, D$.  Note that ${n}_{Q_V}(q) = \eta_i(Q_V)$ if $q = q_i$ in the notation of
$\eqref{solQ}$.
\smallskip

Let the edge set
$$
E \Lambda_{j, D}  =  \{ e_1, e_1^{-1}, e_2, e_2^{-1}, \dots, e_{ | E \Lambda_{j, D} |/2  }, e_{ | E \Lambda_{j, D} |/2  }^{-1}     \}
$$
of the graph $\Lambda_{j, D}$  be indexed as indicated and let $(e_i)_-$ be a vertex of type  \eqref{a3} for every $i$.
\smallskip

We will define the numbers
$\omd(e_i)$ by induction for $i =1,2, \ldots,  | E \Lambda_{j, D} |/2$  by the following procedure which also assigns
intermediate weights $\om_{j,D}(q)$ to vertices $q \in R_V$ of $\Lambda_{j, D}$.
\smallskip

Originally, we set
$$
\om_{j,D}(q) :=  {n}_{Q_V}(q)
$$
for every $q \in R_V$.  For each $i \ge 1$, if the edge $e_i$ goes from $q_1$ to $q_2$ then we set
$$
\omd(e_i) := \min(\omd(q_1), \omd(q_2) )
$$
and redefine the weights of $q_1$ and $q_2$ by setting
\begin{align*}
\omd'(q_1) & := \omd(q_1) - \min(\omd(q_1), \omd(q_2) ) , \\
\omd'(q_2) & := \omd(q_2) - \min(\omd(q_1), \omd(q_2) ) ,
\end{align*}
where $\omd'(q_1) $ denotes the new weight.
\smallskip

 Note that the assignment of a nonnegative weight $\omd(e_i)$ to the edge $e_i$, connecting $q_1$ and $q_2$, can be interpreted as making $\omd(e_i)$ cancellations between terms $\pm  x_{j, D}$ of the subsums
 $$
 \underbrace{ q_1^L + \cdots +q_1^L}_{\text {   $ {n}_{Q_V}(q_1)$ times } }    \quad \mbox{ and} \quad
 \underbrace{  q_2^L + \cdots +q_2^L}_{\text {   $ {n}_{Q_V}(q_2)$ times } }
 $$
 of the sum in the left hand side of the equality
\begin{align}\label{QVr}
 \sum_{q \in Q_V} q^L = -2\brr(Y_1)x_s .
 \end{align}

 Analogously,  the intermediate weight $\omd(q_1)$   of a vertex  $q_1 \in V \Lambda_{j, D}$  can be interpreted as the number of terms  $\pm  x_{j, D}$ of the subsum
 $$
 \underbrace{ q_1^L + \cdots +q_1^L}_{\text {   $ {n}_{Q_V}(q_1)$ times } }
 $$
 which are still uncancelled in  the left hand side of \eqref{QVr}.
 \smallskip

 Therefore, in view of the equality   \eqref{QVr},   in the end of this process, we will obtain that the weights $\omd(q)$ of all vertices $q \in R_V$ are zeros, i.e.,
 cancellations of the terms $\pm x_{j, D}$ are complete, and the weights $\omd(e_i)$ of all edges $e_i$  have desired properties.
\smallskip

It is clear that  the foregoing inductive procedure makes it possible to compute such a weight function $\omd$ in polynomial time {}in the size of the graph $\Lambda_{j, D}$ and in the size of  numbers $n_{Q_V}(q)$, $q \in R_V$, written in binary.  Hence, we can compute weight functions $\omd$ for all pairs $j, D$  in exponential time {}in the size of $Y_1$.

\smallskip

Now we will define the involution $\iota_V$ based on  the weight functions $\omd$.
\smallskip

Let elements of the set $R_V = \{ q_1, \ldots, q_{|R_V|} \}$ be indexed as indicated and  let elements of the combination
\begin{align}\label{QVc}
\begin{split}
Q_V  =  [[ & q_{1,1}, q_{1,2}, \ldots, q_{1,{n}_{Q_V}(q_1)}, \\
& q_{2,1}, q_{2,2}, \ldots, q_{2,{n}_{Q_V}(q_2)}, \\ & \ldots,  \\
   &  q_{|R_V|,1}, q_{|R_V|,2}, \ldots, q_{|R_V|,{n}_{Q_V}(q_{|R_V|})} ]] ,
\end{split}
\end{align}
where $q_{i,\ell} = q_i \in R_V$ for all possible $i, \ell$, be double  indexed  as indicated  according to the indices introduced on elements of $R_V$.
\smallskip

Since the vertices of the graph $Y_{2, Q_V}$  are in bijective correspondence with elements of $Q_V$, see the proof of Lemma~\ref{lem2}, we can also write
$$
V Y_{2, Q_V} = \{ u_{i,\ell} \mid 1 \le i \le |R_V|, \ 1 \le \ell  \le {n}_{Q_V}(q_i) \} ,
$$
where
\begin{align}\label{uqC}
 u_{i, \ell}   \mapsto    q_{i, \ell}
\end{align}
under this correspondence.
\smallskip

Let $q_i \in R_V$ be fixed and  let
$$
q_{m_1(i)}, \ldots, q_{m_{t_i}(i)}
$$
be all the vertices of $\Lambda_{j, D}$, where $ m_1(i) < \cdots < m_{t_i}(i)$, that are connected to $q_i$ by edges $f_1, \ldots, f_{t_i}$, resp.,  in $\Lambda_{j, D}$ with positive weights $\omd(f_1), \ldots, \omd(f_{t_i})$, resp.  We assume that  $q_i$ is the terminal vertex of  the edges $f_1, \ldots, f_{t_i}$.
\smallskip

Recall that $q_i^L$ contains a single term $\pm x_{j, D}$, where the sign is a minus if  $q_i$ has type \eqref{a3}  and the sign is a plus if  $q_i$ has type \eqref{a4}.

\smallskip
According to the weights  $\omd(f_1), \ldots, \omd(f_{t_i})$, we will define $( j, D, i, t)$-blocks of consecutive elements of $Q_V$, see  \eqref{QVc}, in the following manner. The  $(j, D, i, 1)$-block consists of the first $\omd(f_1)$ elements  of the sequence
\begin{align}\label{seqq}
q_{i,1}, q_{i,2}, \ldots, q_{i,{n}_{Q_V}(q_i)} .
\end{align}
The $(j, D, i, 2)$-block consists of the next $\omd(f_2)$ elements  of the sequence \eqref{seqq}  and so on.
The $(j, D, i, t_i)$-block consists of the last $\omd(f_{t_i})$ elements of the sequence \eqref{seqq}.
Since
$$
\sum_{t=1}^{t_i} \omd(f_t) = {n}_{Q_V}(q_i) ,
$$
and $\omd(f_t)  >0$ for every $t$,   it follows that these  $(j, D, i, t)$-blocks, where $t = 1, \dots, t_i$ and
$ j, D, i$  are fixed, will form a partition of the sequence \eqref{seqq} into $t_i$ subsequences.
\smallskip

We  emphasize  that every $(j, D, i, t)$-block is associated with a  vertex $q_i  \in V \Lambda_{j, D} =  R_V$ and with an edge $f_t$ of $\Lambda_{j, D}$
so that $f_t$ ends in $q_i$ and $\omd(f_t) >0$.  In particular, for every $(j, D, i, t)$-block,  associated with a vertex $q_i \in R_V$ and with an edge $f_t$ of  $\Lambda_{j, D}$, we have another $(j, D, i', t')$-block,  associated with a vertex $q_{i'} \in R_V$ and with an edge $f'_{t'}$ of $\Lambda_{j, D}$, so that $q_{i'}  \ne q_i$ and $f'_{t'} = f_t^{-1}$. Here $ f'_1, \ldots, f'_{t'_{i'}}$ are the edges of $\Lambda_{j, D}$ defined for $q_{i'}$ in the same fashion as the edges  $f_1, \ldots, f_{t_i}$ of $\Lambda_{j, D}$ were defined for $q_{i}$.
Note that $i'' = i$ and $f''_{t''} = f_t$ in this notation.
\smallskip

We define the involution $\iota_V$ so that all the  terms   $\pm x_{j, D}$ of the inequalities of the $(j, D, i, t)$-block are mapped by $\iota_V$ to the terms $\mp x_{j, D}$ of the inequalities of the $(j, D, i', t')$-block  in the natural increasing order of elements in the block.
\smallskip

Equivalently, we can say that  the vertices
\begin{align*}
& u_{i,  \omd(f_1) + \cdots + \omd(f_{t-1}) +1},  \\
& u_{i,  \omd(f_1) + \cdots + \omd(f_{t-1}) +2},   \quad   \dots , \\
& u_{i,  \omd(f_1) + \cdots + \omd(f_{t})} ,
\end{align*}
that correspond to $\omd(f_{t})$ inequalities of the  $(j, D, i, t)$-block, see \eqref{uqC},
are connected in $Y_{2, Q_V}$ to the vertices
\begin{align*}
&  u_{i',  \omd(f'_1) + \cdots + \omd(f'_{t'-1}) +1},  \\
 & u_{i',  \omd(f'_1) + \cdots + \omd(f'_{t'-1}) +2},  \quad   \dots , \\
 &  u_{i',  \omd(f'_1) + \cdots + \omd(f'_{t'})},
\end{align*}
resp.,   that correspond to $\omd(f'_{t'})$ inequalities of the  $(j, D, i', t')$-block,
by edges whose labels are $a_j$ if $q_i$ has type \eqref{a3}  or by  edges whose labels are $a_j^{-1}$ if $q_i$ has type \eqref{a4} (here we assume that the edges start at $u_{i, \ell}$ vertices).
It is clear that the foregoing construction of the involution  $\iota_V$ can be done in exponential time.
Therefore, the graph $Y_{2, Q_V}$ can also  be constructed in exponential time, as required.
 Lemma~\ref{lem8} is proved.
\end{proof}

Since the graph $Y_{2,  Q_V}$  can be constructed in exponential time, it follows from Lemma~\ref{lem7} that we can output $ Y_{2,  Q_V}$ as the Stallings graph of the desired subgroup $H_2^*$ of part (b) of Theorem~\ref{th1}.  The proof of Theorem~\ref{th1} is now complete.  \end{proof}

\medskip

It is worthwhile to mention that our construction of the  graph $Y_{2, Q_V}$ is somewhat succinct (cf. the definition of succinct representations of  graphs in \cite{S86}) in the sense that,
despite the fact that the size of $Y_{2, Q_V}$ could be doubly exponential, we are able to give a description of $Y_{2, Q_V}$  in exponential time ({}in the size of $Y_1$).  In particular, the vertices of  $Y_{2, Q_V}$    are represented by exponentially long bit strings and the edges of $Y_{2, Q_V}$  are  drawn in blocks. As a result,
we can find out in exponential time  whether two given vertices of  $Y_{2, Q_V}$  are connected by an edge labelled by  given letter $a_{j}^{\pm 1}$, $a_j \in \A$.

\medskip

\begin{proof}[Proof of Theorem~\ref{pr1}]
(a)  This is immediate from part (c) of Theorem~\ref{th1}.
\smallskip

(b) Suppose that a finitely generated subgroup $H$ of the free group $F = \pi_1(U_m, o_1)$  is not compressed and $K$ is a subgroup of $F$ such that $K$ contains $H$ and  $K$  has a minimal reduced rank $\brr(K)$ such that $\brr(K) < \brr(H)$. Let $Y$ and $ Z$ be reduced $U_m$-graphs of $H$ and $K$, resp.
\smallskip

Since $K$ contains $H$, it follows that there is a locally injective map
 \begin{align}\label{muyz}
\mu : Y \to Z  .
\end{align}

Suppose that $\mu $ is not surjective.  Then there is a subgroup $K'$ of $F$, whose graph is $\mu(Y)$, such that $K'$ contains $H$, $K'$ is a free factor of $K$ and $\brr(K') < \brr(K)$. This inequality, however, contradicts  the minimality of $\brr(K)$.
Hence,  the map  $\mu $ must be surjective. This means that the graph $Z$  with
$$
\brr(Z) = \brr(K) < \brr(H) =\brr(Y)
$$
can be obtained from the graph $Y$ by a finite sequence of operations each of which is identification of two edges whose images are equal in $Z$, i.e.,  the map  \eqref{muyz}  can be factored into a product of  maps that go  through a finite sequence of graphs (not necessarily reduced) each of which is obtained from the previous one by identification of a pair of edges with the same $\ph$-label. Note that the edges of such pairs need not have the same initial or terminal vertices.
\smallskip

Therefore, our  nondeterministic polynomial time algorithm that verifies whether $H$ is {\em not} compressed in $F$ can be run as follows. Starting with the reduced graph $Y$ of $H$,   we   nondeterministically perform a finite sequence of identification of pairs of edges in $Y$ with the same $\ph$-label and, if the resulting $U_m$-graph $Z$ is reduced and satisfies $\brr(Z)  < \brr(Y) $, we accept and  conclude that $H$ is not compressed. Since each single edge identification  decreases the number of edges, it follows that the algorithm runs in nondeterministic linear time {}in the size of $Y$, as required. Theorem~\ref{pr1} is proven.
\end{proof}
\medskip

In conclusion, we mention that it would be of interest  to extend our techniques
to be able to compute the Hanna Heumann coefficient    $\sup_{K} \Big\{ \frac{\brr (H \cap  K)}{\brr (H) \brr (K) } \Big\}$ of a finitely generated noncyclic subgroup $H$ of a free group $F$ and
to algorithmically decide whether $H$ is inert. However, it is not clear at all how to replace in our arguments  the core of the pullback by its single connected component or, alternatively,  guarantee,  at least in certain situations,  that the entire pullback is connected.
\medskip

{\em Acknowledgements.} The author is grateful to the referee for making many helpful remarks and suggestions.

\end{document}